\documentclass[11pt,reqno]{amsart}
\usepackage[a4paper,centering,vscale=0.7,hscale=0.7]{geometry}
\usepackage{mathtools}
\usepackage{amsthm,amssymb}
\usepackage{mathrsfs}

\newtheorem{thm}{Theorem}[section]
\newtheorem{lemma}[thm]{Lemma}
\newtheorem{prop}[thm]{Proposition}

\newtheorem{defi}[thm]{Definition}
\theoremstyle{remark}
\newtheorem{rmk}[thm]{Remark}

\renewcommand{\div}{\operatorname{div}}

\newcommand{\E}{\mathop{{}\mathbb{E}}}
\newcommand{\cF}{\mathscr{F}}
\newcommand{\cL}{\mathscr{L}}
\renewcommand{\P}{\mathbb{P}}
\newcommand{\erre}{\mathbb{R}}

\DeclareMathAlphabet{\mathbbmsl}{U}{bbm}{m}{sl}
\newcommand{\bL}{\mathbbmsl{L}}

\newcommand{\embed}{\hookrightarrow}

\DeclarePairedDelimiter\abs{\lvert}{\rvert}
\DeclarePairedDelimiter\norm{\lVert}{\rVert}
\DeclarePairedDelimiterX\ip[2]{\langle}{\rangle}{#1,#2}

\numberwithin{equation}{section}

\newif\ifbozza
\bozzafalse

\begin{document}

\title[SPDEs in divergence form]{Strong solutions to SPDEs with \\
  monotone drift in divergence form}
\author{Carlo Marinelli}
\address[Carlo Marinelli]{Department of Mathematics, University
      College London, Gower Street, London WC1E 6BT, United
      Kingdom.}
\urladdr{http://goo.gl/4GKJP}
\author{Luca Scarpa}
\address[Luca Scarpa]{Department of Mathematics, University
      College London, Gower Street, London WC1E 6BT, United
      Kingdom.}
\email{luca.scarpa.15@ucl.ac.uk}
\subjclass[2010]{Primary: 60H15, 47H06; Secondary: 46N30}
\keywords{Stochastic evolution equations, singular drift, divergence
  form, multiplicative noise, monotone operators.}
\date{September 6, 2017}

\begin{abstract}
  We prove existence and uniqueness of strong solutions, as well as
  continuous dependence on the initial datum, for a class of fully
  nonlinear second-order stochastic PDEs with drift in divergence
  form. Due to rather general assumptions on the growth of the
  nonlinearity in the drift, which, in particular, is allowed to grow
  faster than polynomially, existing techniques are not applicable. A
  well-posedness result is obtained through a combination of a priori
  estimates on regularized equations, interpreted both as stochastic
  equations as well as deterministic equations with random
  coefficients, and weak compactness arguments. The result is
  essentially sharp, in the sense that no extra hypotheses are needed,
  bar continuity of the nonlinear function in the drift, with respect
  to the deterministic theory.
\end{abstract}

\maketitle


\section{Introduction}
\label{sec:intro}
Let us consider the nonlinear stochastic partial differential equation
\begin{equation}
  \label{eq:0}
  du(t) - \div\gamma(\nabla u(t))\,dt = B(t, u(t))\,dW(t),
  \qquad u(0)=u_0,
\end{equation}
on $L^2(D)$, where $D \subset \erre^n$ is a bounded domain with smooth
boundary. Here $\gamma$ is the gradient of a continuously
differentiable convex function on $\erre^n$ growing faster than
linearly at infinity, the divergence is interpreted in the usual
variational sense, $W$ is a cylindrical Wiener process, and $B$ is a
map with values in the space of Hilbert-Schmidt operators satisfying
suitable Lipschitz continuity hypotheses. Precise assumptions on the
data of the problem are given in {\S}\ref{sec:res} below.

Our main result is the well-posedness of \eqref{eq:0}, in the strong
probabilistic sense, without any polynomial growth condition on
$\gamma$ nor any boundedness assumption on the noise (see
Theorem~\ref{thm:main} below). The lack of growth and coercivity
assumptions on $\gamma$ makes it impossible to apply the variational
approach by Pardoux and Krylov-Rozovski\u{\i} (see \cite{KR-spde,
Pard}), which is the only known general technique to solve nonlinear
stochastic PDEs without linear terms in the drift such as
\eqref{eq:0}, with the possible exception of viscosity solutions, a
theory of which, however, does not seem to be available for such
equations. On the other hand,
we recall that, if $\gamma$ is coercive
and has polynomial growth, the results in \textsl{op.~cit.} provide a
fully satisfactory well-posedness result for \eqref{eq:0}. 

The available literature dealing with stochastic equations in
divergence form such as \eqref{eq:0} is very limited and, to the best
of our knowledge, entirely focused on the case where $\gamma$
satisfies the above-mentioned coercivity and growth assumptions: see,
e.g., \cite{Liu-p} and the bibliography of \cite{LiuRo} for results on
the $p$-Laplace equation, which corresponds to the case
$\gamma(x)=|x|^{p-1}x$, and \cite{luca} on stochastic equations in
divergence form with doubly nonlinear drift.  The main novelty of this
paper is thus to provide a satisfactory well-posedness result in the
strong sense for such divergence-form equations under neither
coercivity nor growth assumptions on $\gamma$.
On the other hand, it is worth recalling that well-posedness results
are available for other classes of monotone SPDEs with nonlinearities
satisfying no coercivity and growth conditions, most notably the
stochastic porous media equation: see, e.g.,
\cite{BDPR-porous}. However, the structure of divergence-form
equations such as \eqref{eq:0} is radically different. Indeed, as is
well-known, the porous media operator is quasilinear, while the
divergence-type operator in \eqref{eq:0} is fully nonlinear. Moreover,
the monotonicity properties (hence the dynamics associated to the the
solutions) are different: the porous media operator is monotone in
$H^{-1}$, whereas the divergence-form operator is monotone in $L^2$.

As is often the case in the treatment of evolution equations of
monotone type, the first step consists in the regularization of
\eqref{eq:0}, replacing $\gamma$ with its Yosida approximation (a
monotone Lipschitz-continuous function), thus obtaining a family of
equations for which well-posedness is known to hold (in our
case, we also need to add a ``small'' elliptic term in the drift as
well as to smooth the diffusion coefficient $B$). In a second step,
one proves that the solutions to the regularized equations are compact
in suitable topologies, so that, by passage to the limit in the
regularization parameters (roughly speaking), a process can be
constructed that, in a final step, is shown to actually be the unique
solution to \eqref{eq:0} and to depend continuously on the initial
datum.  It is well known that the last two steps are the more
challenging ones, and our problem is no exception.

The approach we follow combines elements of the variational method and
\textsl{ad hoc} arguments, most notably a priori estimates on the
solutions to regularized equations, weak compactness techniques, and a
generalized version of It\^o's formula for the square of the norm
under minimal integrability assumptions.  A crucial role is played by
a mix of pathwise and ``averaged''\footnote{That is, in expectation.}
a priori estimates. Even though the approach is reminiscent of that
in \cite{cm:luca}, the problem we consider here is of a completely
different nature, and, correspondingly, new ideas are needed. In
particular, the absence of a linear term in the drift precludes the
possibility of applying a wealth of techniques available for
semi-linear problems. For instance, the strong pathwise compactness
criteria used in \textsl{op.~cit.} are no longer available, so that we
have to rely on weak compactness arguments only. This way one can
construct a limit process, but its identification as a solution
expectedly presents major new issues with respect to the case where
stronger compactness is available.  Moreover, a rather subtle
measurability problem arises from the fact that the divergence is not
injective, which is the reason for assuming $\gamma$ to be a
continuous monotone map, and not just a maximal monotone graph on
$\erre^n \times \erre^n$. A (less regular) solution to the more
general problem when $\gamma$ satisfies only the latter condition will
appear elsewhere. We remark that the results obtained here hold under
hypotheses that are as general as those of the deterministic theory,
except for the continuity assumption on $\gamma$ (see, e.g.,
\cite[pp.~207--ff.]{Barbu:type}).

\medskip

\noindent
\textbf{Acknowledgments.} The authors were partially supported by a
grant of The Royal Society. The first-named author is very grateful to
Prof.~S.~Albeverio for the warm hospitality and the excellent working
conditions at the Interdisziplin\"ares Zentrum f\"ur Komplexe Systeme,
University of Bonn, where parts of this work were written.


\ifbozza\newpage\else\fi
\section{Main result}
\label{sec:res}
Given a positive real number $T$, let
$(\Omega,\cF,(\cF_t)_{t\in[0,T]},\P)$ be a filtered probability space,
fixed throughout, satisfying the so-called ``usual conditions''. We
shall denote a cylindrical Wiener process on a separable Hilbert
space $H$ by $W$.

For any two Hilbert spaces $U$ and $V$, the space of Hilbert-Schmidt
operators from $U$ to $V$ will be denoted by $\cL^2(U,V)$.
Let $D$ be a smooth bounded domain of $\erre^n$, and assume that 
a map
\[
B: \Omega \times [0,T] \times L^2(D) \longrightarrow \cL^2(H,L^2(D))
\]
is given such that, for a constant $C>0$,
\[
\norm[\big]{B(\omega,t,x) - B(\omega,t,y)}_{\cL^2(H,L^2(D))}
\leq C \norm[\big]{x-y}_{L^2(D)}
\]
for all $\omega \in \Omega$, $t \in [0,T]$, $x, y \in L^2(D)$. To
avoid trivial situations, we also assume that, for an
$x_0 \in L^2(D)$, $B(\omega,t,x_0) < C$ for all $\omega$ and $t$. This
implies that $B$ grows at most linearly in $x$, uniformly over
$\omega$ and $t$. Furthermore, the map
$(\omega,t) \mapsto B(\omega,t,x)h$ is assumed to be measurable and
adapted for all $x \in L^2(D)$ and $h \in H$.

We assume that $\gamma$ is the subdifferential of a continuously
differentiable convex function $k:\erre^n \to \erre_+$ such that
$k(0)=0$,
\[
\lim_{\abs{x} \to \infty} \frac{k(x)}{\abs{x}} = + \infty
\]
(i.e. $k$ is superlinear at infinity), and
\[
\limsup_{\abs{x}\to\infty} \frac{k(-x)}{k(x)} < \infty.
\]
Then $\gamma:\erre^n \to \erre^n$ is a continuous maximal monotone
map, i.e.
\[
\bigl(\gamma(x)-\gamma(y)\bigr) \cdot (x-y) \geq 0
\qquad \forall x,y \in \erre^n
\]
(the centered dot stands for the Euclidean scalar product in
$\erre^n$), and (the graph of) $\gamma$ is maximal with respect to the
order by inclusion. Moreover, the convex conjugate function
$k^*:\erre^n \to \erre_+$ of $k$, defined as
\[
k^*(y) = \sup_{r\in\erre^n} \bigl( y \cdot r - k(r) \bigr),
\]
is itself convex and superlinear at infinity. For these facts of
convex analysis, as well as those used in the sequel, we refer to,
e.g., \cite{lema}.

All assumptions on $B$ and $\gamma$ (hence also on $k$) are assumed to
be in force from now on.
\begin{defi}
  Let $u_0$ be an $L^2$-valued $\cF_0$-measurable random variable. A
  \emph{strong solution} to equation \eqref{eq:0} is a process
  $u:\Omega \times [0,T] \to L^2(D)$ satisfying the following
  properties:
  \begin{itemize}
  \item[(i)] $u$ is measurable, adapted and
    \[
    u \in L^1(0,T;W^{1,1}_0(D))
    \]
  \item[(ii)] $B(\cdot,u)h$ is measurable and adapted for all
    $h \in H$ and
    \[
    B(\cdot,u) \in L^2(0,T;\cL^2(H,L^2(D))) \qquad \P\text{-a.s.};
    \]
  \item[(iii)] $\gamma(\nabla u)$ is an $L^1(D)^n$-valued measurable
    adapted process with
    \[
    \gamma(\nabla u) \in L^1(0,T;L^1(D)^n) \qquad \P\text{-a.s.};
    \]
  \item[(iv)] one has, as an equality in $L^2(D)$,
   \begin{equation}
   \label{eq:0i}
   u(t) - \int_0^t{\div \gamma(\nabla u(s))\,ds} 
   = u_0 + \int_0^t B(s,u(s))\,dW(s)
   \qquad \P\text{-a.s.}
   \end{equation}
   for all $t \in [0,T]$.
  \end{itemize}
\end{defi}
Since $\gamma(\nabla u)$ is only assumed to take values in $L^1(D)^n$,
the second term on the left-hand side of \eqref{eq:0i} does not
belong, a priori, to $L^2(D)$. The identity \eqref{eq:0i} has to be
interpreted to hold in the sense of distributions, so that the term
containing $\gamma(\nabla u)$ takes values in $L^2(D)$ by
difference. In fact, the conditions on $B$ in (i) imply that the
stochastic integral in \eqref{eq:0i} is an $L^2(D)$-valued local
martingale.

Let $\mathscr{K}$ be the set of measurable adapted processes $\phi:
\Omega \times [0,T] \to L^2(D)$ such that
\begin{align*}
&\E\sup_{t\leq T} \norm[\big]{\phi(t)}^2_{L^2(D)} 
+ \E\int_0^T \norm[\big]{\phi(t)}_{W^{1,1}_0(D)}\,dt < \infty,\\[4pt]
&\E\int_0^T\!\!\int_D \abs[\big]{\gamma(\nabla\phi(t,x))}\,dx\,dt
< \infty,\\[4pt]
&\E\int_0^T\!\!\int_D \bigl( k(\nabla\phi(t,x)) +%
k^*(\gamma(\nabla\phi(t,x))) \bigr)\,dx\,dt < \infty.
\end{align*}

\medskip

Our main result is the following.
\begin{thm}
  \label{thm:main}
  Let $u_0 \in L^2(\Omega;L^2(D))$ be $\cF_0$-measurable. Then
  \eqref{eq:0} admits a strong solution $u$, which is unique within
  $\mathscr{K}$. Moreover, $u$ has weakly continuous paths in $L^2(D)$
  and the solution map $u_0 \mapsto u$ is Lipschitz-continuous from
  $L^2(\Omega;L^2(D))$ to $L^2(\Omega;L^\infty(0,T;L^2(D)))$.
\end{thm}

We do not know whether well-posedness continues to hold also without
the condition that the solution belongs to $\mathscr{K}$. This
assumption, in fact, plays a crucial role in the proof of uniqueness.

\medskip

Abbreviated notation for function spaces will be used from now on:
Lebesgue and Sobolev spaces on $D$ will be denoted without explicit
mention of $D$ itself; for any $p \in [1,\infty]$, $L^p(\Omega)$ will
be denoted by $\bL^p$, $L^p(0,T)$ by $L^p_t$, and $L^p(D)$ sometimes
by $L^p_x$. Mixed-norm spaces will be denoted just by juxtaposition,
e.g. $\bL^pL^q_tL^r_x$ to mean $L^p(\Omega;L^q(0,T;L^r(D)))$ and
$L^1_{t,x}$ to mean $L^1([0,T] \times D)$.


\ifbozza\newpage\else\fi
\section{An It\^o formula for the square of the norm}
\label{sec:aux}
We prove an It\^o formula for the square of the $L^2$-norm of a class
of processes with minimal integrability conditions. This is an
essential tool to prove uniqueness of strong solutions and their
continuous dependence on the initial datum in Sections \ref{sec:add}
and \ref{sec:pf} below, and it is interesting in its own right.
\begin{prop}
  \label{prop:ito}
  Assume that
  \[
  y(t) + \alpha\int_0^t y(s)\,ds
  - \int_0^t \div \zeta(s) \,ds =
  y_0 + \int_0^t C(s)\,dW(s)
  \]
  holds in $L^2$ for all $t \in [0,T]$ $\P$-a.s., where
  $\alpha \geq 0$ is a constant,
  \[
  y:\Omega \times [0,T] \to L^2, \qquad \zeta:\Omega \times [0,T] \to L^1,
  \qquad C: \Omega \times [0,T] \to \cL^2(H,L^2)
  \]
  are measurable adapted processes such that
  \[
  y \in \bL^2 L^\infty_t L^2_x \cap \bL^1 L^1_t W^{1,1}_0,
  \qquad \zeta \in \bL^1 L^1_{t,x},
  \qquad C \in \bL^2 L^2_t \cL^2(H,L^2),
  \]
  and $y_0$ is an $\cF_0$-measurable $L^2$-valued random variable with
  $\E\norm{y_0}^2<\infty$. If there exists a constant $c>0$ such that
  \[
  \E\int_0^T\!\!\int_D \bigl(k(c\nabla y)
  + k^*(c\zeta)\bigr) < \infty,
  \]
  then 
  \begin{align*}
    &\frac12 \norm{y(t)}^2 
    + \alpha \int_0^t\norm{y(s)}^2\,ds
    + \int_0^t\!\!\int_D \zeta(s,x) \cdot \nabla y(s,x)\,dx\,ds\\
    &\hspace{3em} = \frac12\norm{y_0}^2
    + \frac12 \int_0^t \norm[\big]{C(s)}^2_{\cL^2(H,L^2)}\,ds
    + \int_0^t y(s) C(s)\,dW(s)
  \end{align*}
  for all $t\in[0,T]$ $\P$-almost surely.
\end{prop}
\begin{proof}
  Note that $\div\zeta \in (W^{1,\infty}_0)'$, hence, by Sobolev
  embedding theorems and duality, there exists a positive integer $r$
  such that $\div\zeta \in H^{-r}$. Therefore, denoting the Dirichlet
  Laplacian on $L^2(D)$ by $\Delta$, there also exists a positive
  integer $m$ such that $(I-\delta\Delta)^{-m}$, $\delta>0$, maps
  $H^{-r}$ and (a fortiori) $L^2$ to $H^1_0 \cap W^{1,\infty}$. Using
  the notation $h^\delta:=(I-\delta\Delta)^{-m}h$, it is readily seen
  that
  \[
  y^\delta(t)+\alpha\int_0^ty^\delta(s)\,ds-\int_0^t\div\zeta^\delta(s)\,ds=
  y^\delta_0+\int_0^tT^\delta(s)\,dW(s)
  \]
  for all $t\in[0,T]$ $\P$-a.s. as an identity in $L^2$, for which
  It\^o's formula yields
  \begin{align*}
  &\frac12 \norm[\big]{y^\delta(t)}^2 
  + \alpha \int_0^t \norm[\big]{y^\delta(s)}^2\,ds
  + \int_0^t\!\!\int_D \zeta^\delta \cdot \nabla y^\delta\\
  &\hspace{3em} = \frac12 \norm[\big]{y^\delta_0}^2
  + \frac12 \int_0^t \norm[\big]{C^\delta(s)}^2_{\cL^2(H,L^2)}\,ds
    +\int_0^t y^\delta(s) C^\delta(s)\,dW(s)
  \end{align*}
  for all $t \in [0,T]$ $\P$-almost surely. We are going to pass to
  the limit as $\delta \to 0$ in this identity. The dominated
  convergence theorem immediately implies that, $\P$-a.s.,
  \begin{align*}
  \norm[\big]{y^\delta(t)}^2 &\longrightarrow \norm[\big]{y(t)}^2,\\
  \int_0^t \norm[\big]{y^\delta(s)}^2\,ds &\longrightarrow
  \int_0^t \norm[\big]{y(s)}^2\,ds,\\
  \int_0^t \norm[\big]{C^{\delta}(s)}_{\cL^2(H,L^2)}^2\,ds
  &\longrightarrow \int_0^t \norm[\big]{C(s)}_{\cL^2(H,L^2)}^2\,ds
  \end{align*}
  for all $t \in [0,T]$, and $\norm{y^\delta_0}^2 \to \norm{y_0}^2$,
  as $\delta \to 0$.
  Defining the real local martingales
  \[
  M^{\delta} := (y^{\delta} C^{\delta}) \cdot W, \qquad
  M := (yC) \cdot W,
  \]
  we are going to show that
  \[
  \E\sup_{t\leq T} \abs[\big]{M^\delta(t) - M(t)} \longrightarrow 0
  \]
  as $\delta \to 0$. In fact, Davis' inequality for local martingales
  (see, e.g., \cite{cm:expo16}) yields
  \begin{align*}
  \E\sup_{t\leq T} \abs[\big]{M^\delta(t) - M(t)} &\lesssim
  \E\bigl[M^\delta-M,M^\delta-M\bigr]_T^{1/2}\\
  &=\E\biggl(\int_0^T%
    \norm[\big]{y^\delta(t) C^\delta(t) - y(t)C(t)}^2_{\cL^2(H,\erre)}\,dt
    \biggr)^{1/2},
  \end{align*}
  and one has, identifying $\cL^2(H,\erre)$ with $H$ and recalling
  that $(I-\delta\Delta)^{-m}$ is contractive in $L^2$,
  \begin{align*}
  \norm[\big]{y^\delta C^\delta - yC}_H &\leq 
  \norm[\big]{y^\delta C^\delta - y^\delta C}_H 
  + \norm[\big]{y^\delta C - y C}_H\\
  &\leq \Bigl( \sup_{t\leq T} \norm{y(t)} \Bigr)
        \norm[\big]{C^\delta-C}_{\cL^2(H,L^2)}
  + \norm[\big]{y^\delta C - y C}_H,
  \end{align*}
  so that
  \begin{align*}
  &\E\biggl(\int_0^T%
    \norm[\big]{y^\delta(t) C^\delta(t) - y(t)C(t)}^2_H\,dt\biggr)^{1/2}\\
  &\hspace{3em} \lesssim \E \sup_{t\leq T} \norm{y(t)}
    \biggl( \int_0^T \norm[\big]{C^\delta(t)-C(t)}^2_{\cL^2(H,L^2)}\,dt
    \biggr)^{1/2}\\
  &\hspace{3em}\quad + \E\biggl( \int_0^T
    \norm[\big]{(y^\delta(t)-y(t))C(t)}^2_H\,dt \biggr)^{1/2}.
  \end{align*}
  It follows by the Cauchy-Schwarz inequality that the first term on
  the right-hand side is dominated by
  \[
  \Bigl( \E \sup_{t\leq T} \norm{y(t)}^2 \Bigr)^{1/2} 
  \biggl( \E \int_0^T
  \norm[\big]{C^\delta(t)-C(t)}^2_{\cL^2(H,L^2)}\,dt \biggr)^{1/2},
  \]
  which converges to zero by properties of Hilbert-Schmidt operators
  and the dominated convergence theorem. Moreover,
  \[
  \norm[\big]{(y^\delta(t)-y(t))C(t)}^2_H \lesssim 
  \norm[\big]{y(t)}^2 \norm[\big]{C(t)}^2_{\cL^2(H,L^2)}
  \]
  and $y \in L^\infty_t L^2_x$, $C \in L^2_t\cL(H,L^2_x)$ $\P$-a.s.
  imply, by dominated convergence, that
  \[
  \int_0^T \norm[\big]{(y^\delta(t)-y(t))C(t)}^2_H\,dt \longrightarrow 0
  \]
  $\P$-a.s. as $\delta \to 0$. Since
  \[
  \biggl( \int_0^T \norm[\big]{(y^\delta(t)-y(t))C(t)}^2_H\,dt
  \biggr)^{1/2} \lesssim \sup_{t\leq T} \norm{y(t)}
  \biggl( \int_0^T \norm[\big]{C(t)}^2_{\cL^2(H,L^2)}\,dt \biggr)^{1/2}
  \]
  and, by the Cauchy-Schwarz inequality,
  \begin{align*}
  &\E\sup_{t\leq T} \norm{y(t)}
  \biggl( \int_0^T \norm[\big]{C(t)}^2_{\cL^2(H,L^2)}\,dt \biggr)^{1/2}\\
  &\hspace{3em} \leq \Bigl( \E\sup_{t\leq T} \norm{y(t)}^2 \Bigr)^{1/2}
  \biggl( \E\int_0^T \norm[\big]{C(t)}^2_{\cL^2(H,L^2)}\,dt \biggr)^{1/2}
  < \infty,
  \end{align*}
  again by dominated convergence it follows that
  \[
  \E \biggl( \int_0^T \norm[\big]{(y^\delta(t)-y(t))C(t)}^2_H\,dt
  \biggr)^{1/2} \longrightarrow 0
  \]
  as $\delta \to 0$. We have thus shown that
  $\E\sup_{t\leq T} \abs[\big]{M^\delta(t)-M(t)} \to $ as
  $\delta \to 0$, hence, in particular, that
  \[
  \int_0^t y^\delta(s)C^\delta(s)\,dW(s) \longrightarrow 
  \int_0^t y(s)C(s)\,dW(s)
  \]
  in probability as $\delta \to 0$ for all $t \in [0,T]$.

  To complete the proof, we are going to show that
  $\nabla Y^{\delta}\cdot \zeta^{\delta} \to \nabla Y\cdot \zeta$ in
  $\bL^1L^1_{t,x}$, which readily implies that
  \[
  \int_0^t\!\!\int_D \nabla y^{\delta}(s,x) \cdot \zeta^{\delta}(s,x) \,dx\,ds
  \longrightarrow
  \int_0^t\!\!\int_D \nabla y(s,x) \cdot \zeta(s,x) \,dx\,ds
  \]
  in probability for all $t \in [0,T]$. Since
  $\nabla y^{\delta} \to \nabla y$ and $\zeta^{\delta} \to \zeta$ in
  measure in $\Omega \times (0,T) \times D$, in view of Vitali's
  theorem, it suffices to prove that the sequence
  $(\nabla y^{\delta}\cdot \zeta^{\delta})$ is uniformly integrable in
  $\Omega \times (0,T) \times D$. One has
  \begin{align*}
  c^2 \bigl( \nabla y^\delta \cdot \zeta^\delta \bigr) &\leq
  k\bigl( c\nabla y^\delta \bigr) + k^*\bigl( c\zeta^\delta \bigr),\\
  -c^2 \bigl( \nabla y^\delta \cdot \zeta^\delta \bigr) &\leq
  k\bigl( c(-\nabla y^\delta) \bigr) + k^*\bigl( c\zeta^\delta \bigr)
  \end{align*}
  hence
  \begin{align*}
  c^2 \abs[\big]{\nabla y^\delta \cdot \zeta^\delta} &\lesssim
  k\bigl( c\nabla y^\delta \bigr) + k\bigl( c(-\nabla y^\delta) \bigr)
  + k^*\bigl( c\zeta^\delta \bigr)\\
  &\lesssim 1 + k\bigl( c\nabla y^\delta \bigr)
  + k^*\bigl( c\zeta^\delta \bigr),
  \end{align*}
  where the second inequality follows by the hypothesis
  $\limsup_{\abs{x}\to\infty} k(-x)/k(x)<\infty$.  By Jensen's
  inequality for sub-Markovian operators (see
  \cite[Theorem~3.4]{Haa07}) we also have
  \begin{align*}
  k\bigl( c\nabla y^\delta \bigr) 
  &= k\bigl( (I-\delta\Delta)^{-m} c\nabla y \bigr)
  \leq (I-\delta\Delta)^{-m} k\bigl( c\nabla y \bigr),\\
  k^*\bigl( c\zeta^\delta \bigr)
  &=k^*\bigl( (I-\delta\Delta)^{-m} c\zeta \bigr)
  \leq (I-\delta\Delta)^{-m}   k^*\bigl( c\zeta \bigr),
  \end{align*}
  hence
  \[
  c^2 \abs[\big]{\nabla y^\delta \cdot \zeta^\delta} \lesssim
  1 + (I-\delta\Delta)^{-m} \bigl( k(c\nabla y) + k^*(c\zeta) \bigr),
  \]
  where the right-hand side is uniformly integrable because it converges
  in $\bL^1L^1_{t,x}$ as $\delta \to 0$. This yields that
  $(\nabla y^\delta \cdot \zeta^\delta)$ is uniformly integrable as
  well, thus concluding the proof.
\end{proof}


\ifbozza\newpage\else\fi
\section{Well-posedness for an auxiliary SPDE}
\label{sec:reg}
Let $V_0$ be a separable Hilbert space, densely and continuously
embedded\footnote{Continuous embedding of a Banach space $E$ in a
  Banach space $F$ will be denoted by $E \embed F$.} in $H^1_0$, and
continuously embedded in $W^{1,\infty}$.  The Sobolev embedding
theorem easily implies that such a space exists indeed.

We are going to prove that the auxiliary equation
\begin{equation}
\label{eq:V0}
du(t) - \div \gamma(\nabla u(t))\,dt = G(t)\,dW(t),
\qquad u(0)=u_0,
\end{equation}
where $G$ is an $\cL^2(U,V_0)$-valued process, is well posed.

\begin{prop}
  \label{prop:V0}
  Assume that $u_0 \in \bL^2(L^2)$ is $\cF_0$-measurable and that
  $G:\Omega \times [0,T] \to \cL^2(U,V_0)$ is measurable and adapted, with
  \[
  \E\int_0^T \norm[\big]{G(t)}^2_{\cL^2(U,V_0)}\,dt < \infty.
  \]
  Then equation \eqref{eq:V0} admits a unique strong solution
  $u$ such that
  \begin{align*}
    &\E\sup_{t\leq T} \norm{u(t)}^2 
      + \E\int_0^T \norm[\big]{u(t)}_{W^{1,1}_0}\,dt < \infty,\\
    &\E\int_0^T \norm[\big]{\gamma(\nabla u(t))}_{L^1} \,dt < \infty,\\
    & \int_0^T \bigl( \norm[\big]{k(\nabla u(t))}_{L^1} 
      + \norm[\big]{k^*(\gamma(\nabla u(t)))}_{L^1} \,dt \bigr) 
      < \infty \quad \text{$\P$-almost surely}.
  \end{align*}
  Moreover, the paths of $u$ are $\P$-a.s. weakly continuous with
  values in $L^2$.
\end{prop}
The assumptions of Proposition \ref{prop:V0} are (tacitly) assumed to
hold throughout the section.

\medskip

Let $\gamma_\lambda: \erre^n \to \erre^n$, $\lambda>0$, be the Yosida
regularization of $\gamma$, i.e.
\[
\gamma_\lambda := \frac{1}{\lambda}\bigl( I - (I+\lambda\gamma)^{-1} \bigr),
\qquad \lambda>0,
\]
and consider the regularized equation
\[
  du_\lambda(t) - \div \gamma_\lambda (\nabla u_\lambda(t))\,dt -
  \lambda\Delta u_\lambda(t)\,dt = G(t)\,dW(t), \qquad
  u_\lambda(0)=u_0.
\]
Since $\gamma_\lambda$ is monotone and Lipschitz-continuous, it is not
difficult to check that the operator
\[
v \longmapsto - \bigl( \div\gamma_\lambda(\nabla v) + \lambda\Delta v \bigr)
\]
satisfies the conditions of the classical variational approach by
Pardoux, Krylov and Rozovski\u{i} \cite{KR-spde,Pard} on the Gelfand
triple $H^1_0 \embed L^2 \embed H^{-1}$, hence there exists a unique
adapted process $u_\lambda$ with values in $H^1_0$ such that
\[
\E \norm[\big]{u_\lambda}^2_{C_t L^2_x} 
+ \E\int_0^T \norm[\big]{u_\lambda(t)}^2_{H^1_0}\,dt < \infty
\]
and
\begin{equation}
  \label{eq:reg}
  u_\lambda(t) - \int_0^t \div\gamma_\lambda(\nabla u_\lambda(s))\,ds
  - \lambda\int_0^t \Delta u_\lambda(s)\,ds = u_0 + \int_0^t G(s)\,dW(s)  
\end{equation}
in $H^{-1}$ for all $t \in [0,T]$.

\subsection{A priori estimates}
We are now going to establish several a priori estimates for
$u_\lambda$ and related processes, both pathwise and in expectation.

We begin with a simple maximal estimate for stochastic integrals that
will be used several times in the sequel.
\begin{lemma}  \label{lm:DY}
  Let $U$, $H$, $K$ be separable Hilbert spaces. If
  \[
  F: \Omega \times [0,T] \to \cL(H,K), \qquad
  G: \Omega \times [0,T] \to \cL^2(U,H)
  \]
  are measurable and adapted processes such that
  \[
  \E\sup_{t\leq T} \norm[\big]{F(t)}^2_{\cL(H,K)}
  + \E\int_0^T \norm[\big]{G(t)}^2_{\cL^2(U,H)} \,dt < \infty,
  \]
  then, for any $\varepsilon>0$,
  \begin{align*}
  &\E\sup_{t\leq T} \norm[\bigg]{\int_0^t F(s)G(s)\,dW(s)}_K\\
  &\hspace{3em} \leq \varepsilon
  \E\sup_{t\leq T} \norm[\big]{F(t)}^2_{\cL(H,K)}
  + N(\varepsilon) \E\int_0^T \norm[\big]{G(t)}^2_{\cL^2(U,H)}\,dt.
  \end{align*}
\end{lemma}
\begin{proof}
  By the ideal property of Hilbert-Schmidt operators (see, e.g.,
  \cite[p.~V.52]{Bbk:EVT}), one has
  \begin{align*}
  \norm[\big]{F(s)G(s)}_{\cL^2(U,K)} &\leq
  \norm[\big]{F(s)}_{\cL(H,K)} \norm[\big]{G(s)}_{\cL^2(U,H)}\\
  &\leq \sup_{s\leq T} \norm[\big]{F(s)}_{\cL(H,K)} \norm[\big]{G(s)}_{\cL^2(U,H)}
  \end{align*}
  for all $s \in [0,T]$, hence
  \[
  \int_0^T \norm[\big]{F(s)G(s)}^2_{\cL^2(U,K)}\,ds \leq
  \sup_{s\leq T} \norm[\big]{F(s)}^2_{\cL(H,K)} 
  \int_0^T \norm[\big]{G(s)}^2_{\cL^2(U,H)}\,ds,
  \]
  where the right-hand side is finite $\P$-a.s. thanks to the
  assumptions on $F$ and $G$. Then $(FG) \cdot W$ is a $K$-valued
  local martingale, for which Davis' inequality yields
  \begin{align*}
  \E\sup_{t\leq T} \norm[\bigg]{\int_0^t F(s)G(s)\,dW(s)}_K
  &\lesssim 
  \E\bigl[(FG)\cdot W,(FG)\cdot W\bigr]_T^{1/2}\\
  &= \E\biggl( \int_0^T 
  \norm[\big]{F(s)G(s)}^2_{\cL^2(U,K)}\,ds \biggr)^{1/2}\\
  &\leq \E \sup_{s\leq T} \norm[\big]{F}_{\cL(H,K)}
  \biggl( \int_0^T \norm[\big]{G(s)}^2_{\cL^2(U,H)}\,ds \biggr)^{1/2}.
  \end{align*}
  The proof is finished invoking the elementary inequality
  \[
  ab \leq \frac12 \bigl( \varepsilon a^2 + \frac{1}{\varepsilon} b^2\bigr)
  \qquad \forall a, b \in \erre, \; \varepsilon>0,
  \]
  and choosing $\varepsilon$ properly.
\end{proof}
\noindent The estimate in the previous lemma will be used only in the
case $K=\erre$. The more general proof we have given is not more
complicated than in the simpler case actually needed.

\begin{lemma}  \label{lm:aspetta}
  There exists a constant $N$ such that
  \begin{align*}
    &\norm[\big]{u_\lambda}_{\bL^2 C_t L^2_x} + 
    \lambda^{1/2} \norm[\big]{\nabla u_\lambda}_{\bL^2 L^2_{t,x}} +
    \norm[\big]{\gamma_\lambda(\nabla u_\lambda)%
      \cdot \nabla u_\lambda}_{\bL^1 L^1_{t,x}}\\
    &\hspace{3em} < N \Bigl( \norm[\big]{u_0}_{\bL^2 L^2_x} 
      + \norm[\big]{G}_{\bL^2 L^2_t \cL^2(H,L^2_x)} \Bigr).
  \end{align*}
\end{lemma}
\begin{proof}
  It\^o's formula yields
  \begin{align*}
  &\norm[\big]{u_\lambda(t)}^2 
  + 2\int_0^t\!\!\int_D \gamma(\nabla u_\lambda(s)) \cdot %
    \nabla u_\lambda(s)\,dx\,ds
  + 2\lambda \int_0^t \norm[\big]{\nabla u_\lambda(s)}^2\,ds\\
  &\hspace{3em} = \norm[\big]{u_0}^2
  + 2\int_0^t u_\lambda(s) G(s)\,dW(s) 
  + \frac12 \int_0^t \norm[\big]{G(s)}^2_{\cL^2(H,L^2)}\,ds,
  \end{align*}
  where $u_\lambda$ in the stochastic integral on the right-hand side
  has to be interpreted as taking values in $\cL(L^2,R) \simeq L^2$. 
  Taking supremum in time and expectation we get
  \begin{align*}
  &\E\norm[\big]{u_\lambda}^2_{C_t L^2_x} 
  + \E\int_0^T\!\!\int_D \gamma_\lambda(\nabla u_\lambda(s))%
     \cdot \nabla u_\lambda(s)\,dx\,ds
  + \lambda\E\norm[\big]{\nabla u_\lambda}^2_{L^2_{t,x}}\\
  &\hspace{5em} \lesssim
  \E\norm[\big]{u_0}^2 + \E\norm[\big]{G}^2_{L^2_t\cL^2(H,L^2))}
  + \E\sup_{t\in[0,T]} \abs[\bigg]{\int_0^t u_\lambda(s) G(s)\,dW(s)},
  \end{align*}
  where, by Lemma \ref{lm:DY},
  \[
  \E\sup_{t\in[0,T]} \abs[\bigg]{\int_0^t u_\lambda(s) G(s)\,dW(s)}
  \leq \varepsilon \E \norm[\big]{u_\lambda}^2_{C_t L^2_x}
  + N(\varepsilon) \E\int_0^T \norm[\big]{G(s)}^2_{\cL^2(H,L^2)}\,ds
  \]
  for any $\varepsilon>0$. The proof is completed choosing
  $\varepsilon$ small enough and recalling that $\gamma_\lambda$ is
  monotone.
\end{proof}

\begin{lemma}
  \label{lm:L1a}
  The families $(\nabla u_\lambda)$ and
  $(\gamma_\lambda(\nabla u_\lambda))$ are relatively weakly compact
  in $\bL^1L^1_{t,x}$.
\end{lemma}
\begin{proof}
  Recall that, for any $y$, $r \in \erre^n$, ones has
  $k(y)+k^*(r)=r\cdot y$ if and only if
  $r \in \partial k(y)=\gamma(y)$. Therefore, since
  \[
  \gamma_\lambda(x) \in \partial k\bigl((I+\lambda\gamma)^{-1}x\bigr)
  = \gamma\bigl((I+\lambda\gamma)^{-1}x\bigr) \qquad \forall x \in \erre^n,
  \]
  we deduce, by the definition of $\gamma_\lambda$, that
  \begin{align}
  \nonumber
  k\bigl((I+\lambda\gamma)^{-1}x\bigr) + k^*\bigl(\gamma_\lambda(x)\bigr)
  &= \gamma_\lambda(x)\cdot(I+\lambda\gamma)^{-1}x \\
  \label{eq:pluto}
  &=\gamma_\lambda(x)\cdot x-\lambda\abs[\big]{\gamma_\lambda(x)}^2
  \leq \gamma_\lambda(x)\cdot x
  \qquad \forall x \in \erre^n.
  \end{align}
  By Lemma \ref{lm:aspetta} we infer that there exists a
  constant $N$, independent of $\lambda$, such that
  \[
  \E\int_0^T\!\!\int_D k^*\bigl(\gamma_\lambda(\nabla u_\lambda)\bigr)
  \leq \E\int_0^T\!\!\int_D \gamma_\lambda(\nabla u_\lambda)\cdot \nabla u_\lambda
  < N.
  \]
  Since $k^*$ is superlinear at infinity, the family
  $(\gamma_\lambda(\nabla u_\lambda))$ is uniformly integrable on
  $\Omega \times (0,T) \times D$ by the de la Vall\'ee Poussin
  criterion (see the appendix), hence relatively weakly
  compact in $\bL^1L^1_{t,x}$ by a well-known theorem of Dunford and
  Pettis.

  Similarly, Lemma \ref{lm:aspetta} and \eqref{eq:pluto} imply that
  there exists a constant $N$, independent of $\lambda$, such that
  \[
  \E\int_0^T\!\!\int_D k\bigl((I+\lambda\gamma)^{-1}\nabla u_\lambda\bigr)
  \leq \E\int_0^T\!\!\int_D \gamma_\lambda(\nabla u_\lambda) \cdot%
    \nabla u_\lambda < N.
  \]
  Since $k$ is superlinear at infinity, the criteria by de la
  Vall\'ee Poussin and Dunford-Pettis imply that the sequence
  $(I+\lambda\gamma)^{-1}\nabla u_\lambda$ is uniformly integrable on
  $\Omega \times (0,T) \times D$, hence relatively weakly compact in
  $\bL^1 L^1_{t,x}$. Moreover, since
  \[
  \nabla u_\lambda=(I+\lambda\gamma)^{-1}\nabla
  u_\lambda+\lambda\gamma_\lambda(\nabla u_\lambda),
  \]
  the relative weak compactness of $(\nabla u_\lambda)$ immediately
  follows by the same property of $(\gamma_\lambda(\nabla u_\lambda))$
  proved above.
\end{proof}

We shall need below the following classical absolute continuity
result, whose proof can be found, for instance, in
\cite[p.~25]{Barbu:type}.
\begin{lemma}
  \label{lm:AC}
  Let $V$ and $H$ be Hilbert spaces with $V \embed H \embed
  V'$. Assume that $u \in L^2(a,b;V)$ and $u' \in L^2(a,b;V')$, where
  $u'$ is the derivative of $u$ in the sense of $V'$-valued
  distributions. Then there exists $\tilde{u} \in C([a,b];H)$ such
  that $u(t)=\tilde{u}(t)$ for almost all $t \in [a,b]$. Moreover, for
  any $v$ satisfying the same hypotheses of $u$, $\ip{u}{v}$ is
  absolutely continuous on $[a,b]$ and
  \[
  \frac{d}{dt} \ip[\big]{u(t)}{v(t)} = \ip[\big]{u'(t)}{v(t)} +
  \ip[\big]{u(t)}{v'(t)}.
  \]
\end{lemma}
\noindent As customary, both the duality pairing between $V$ and $V'$ as
well as the scalar product of $H$ have been denoted by the same symbol.

\medskip

From now on we shall assume, without loss of generality, that
$\lambda \in \mathopen]0,1\mathclose]$.
\begin{lemma}
  \label{lm:stp}
  There exists $\Omega' \subseteq \Omega$ with $\P(\Omega')=1$ and
  $M:\Omega' \to \erre$ such that
  \[
  \norm[\big]{u_\lambda(\omega)}_{L^\infty_tL^2_x}
  + \sqrt{\lambda} \, \norm[\big]{\nabla u_\lambda(\omega)}_{L^2_{t,x}}
  + \norm[\big]{k_\lambda(\nabla u_\lambda(\omega))}_{L^1_{t,x}} < M(\omega)
  \]
  for all $\omega \in \Omega'$.
\end{lemma}
\begin{proof}
  Setting $v_\lambda := u_\lambda - G \cdot W$, equation
  \eqref{eq:reg} can be written as
  \[
  v_\lambda(t) - \int_0^t \div\bigl( \gamma_\lambda(\nabla
  u_\lambda(s)) + \lambda\nabla u_\lambda(s)\bigr)\,ds = u_0,
  \]
  or, equivalently, as
  \begin{equation}
    \label{eq:vul}
    v'_\lambda - \div \bigl( \gamma_\lambda(\nabla u_\lambda)
    + \lambda\nabla u_\lambda \bigr) = 0, \qquad v_\lambda(0) = u_0.
  \end{equation}
  By It\^o's isometry and Doob's inequality, one has
  \[
  \E\sup_{t\leq T} \norm[\bigg]{\int_0^t G(s)\,dW(s)}^2_{V_0} \lesssim
  \E\int_0^T \norm[\big]{G(s)}^2_{\cL(H,V_0)}\,ds < \infty,
  \]
  hence $G \cdot W \in \bL^2 L^\infty_t H^1_0$, because
  $V_0 \embed H^1_0$. In particular, since
  $u_\lambda \in \bL^2 L^\infty_t H^1_0$, it follows that
  $v_\lambda \in \bL^2 L^\infty_t H^1_0$. Moreover, since
  $\div\gamma_\lambda(\nabla u_\lambda)$ and $\Delta u_\lambda$ belong
  to $\bL^2 L^2_t H^{-1}$, by the previous identity we also deduce that
  $v'_\lambda(\omega) \in L^2_t H^{-1}$ for $\P$-a.a.
  $\omega \in \Omega$. In particular, taking into account the
  hypotheses on $u_0$ and $G$, there exists $\Omega' \subset \Omega$,
  with $\P(\Omega')=1$, such that
  \begin{gather*}
  u_0(\omega) \in L^2_x, 
  \quad G \cdot W(\omega,\cdot) \in L^\infty_t V_0,\\
  v_\lambda(\omega) \in L^2_t H^1_0, \quad 
  v'_\lambda(\omega) \in L^2_t H^{-1}
  \end{gather*}
  for all $\omega \in \Omega'$. Let us consider from now on a fixed
  but arbitrary $\omega \in \Omega'$. Taking the duality pairing of
  \eqref{eq:vul} by $v_\lambda$ and integrating (more precisely,
  applying Lemma~\ref{lm:AC}) implies that, for all $t \in [0,T]$,
  \begin{align*}
  &\frac12\norm{v_\lambda(t)}^2 + \int_0^t\!\!\int_D
  \gamma_\lambda(\nabla u_\lambda(s)) \cdot \nabla v_\lambda(s)\,dx\,ds\\
  &\hspace{3em}
  + \lambda\int_0^t\!\!\int_D \nabla u_\lambda(s) \cdot%
     \nabla v_\lambda(s)\,dx\,ds
  = \frac12\norm{u_0}^2,
  \end{align*}
  where $\norm{u_\lambda} \leq \norm{v_\lambda} + \norm{G\cdot W}$,
  hence
  $\norm{u_\lambda}^2 \leq 2\bigl(\norm{v_\lambda}^2 + \norm{G\cdot
    W}^2\bigr)$, as well as
  \[
  \norm{v_\lambda}^2 \geq \frac12 \norm{u_\lambda}^2 -
  \norm{G \cdot W}^2.
  \]
  Moreover, Young's inequality yields
  \begin{align*}
    \int_D\nabla u_\lambda\cdot\nabla v_\lambda &=
    \norm[\big]{\nabla u_\lambda}^2 
    - \int_D{\nabla u_\lambda\cdot\nabla(G\cdot W)}\\
    &\geq \frac{1}{2} \norm[\big]{\nabla u_\lambda}^2
     -\frac{1}{2} \norm[\big]{\nabla(G\cdot W)}^2,
  \end{align*}
  hence also, taking into account the previous estimate,
  \begin{equation}
  \label{eq:pippo}
  \begin{split}
  &\frac12 \norm[\big]{u_\lambda(t)}^2 
  + 2 \int_0^t\!\!\int_D  \gamma_\lambda(\nabla u_\lambda(s)) \cdot
      \nabla v_\lambda(s)\,dx\,ds
  + \lambda \int_0^t \norm[\big]{\nabla u_\lambda(s)}^2\,ds\\
  &\hspace{3em} \leq \norm[\big]{u_0}^2 + \norm[\big]{G\cdot W(t)}^2
  + \lambda \int_0^t \norm[\big]{\nabla(G\cdot W(s))}^2\,ds.
  \end{split}
  \end{equation}
  Let $k_\lambda$ be the Moreau-Yosida regularization of $k$, i.e.
  \[
  k_\lambda(x) := \inf_{y \in \erre^n} \Bigl( 
    k(y) + \frac{\abs{x-y}^2}{2\lambda} \Bigr), \qquad \lambda>0.
  \]
  As is well known, $k_\lambda$ is a proper convex function that
  converges pointwise to $k$ from below, and
  $\partial k_\lambda = \gamma_\lambda$. Therefore, it follows from
  \[
  \gamma_\lambda(x)\cdot(x-y) \geq k_\lambda(x)-k_\lambda(y)
  \geq k_\lambda(x)-k(y) \qquad \forall x,y \in \erre^n
  \]
  that
  \begin{align*}
    &\int_0^t\!\!\int_D \gamma_\lambda(\nabla u_\lambda(s)) \cdot
      \nabla v_\lambda(s)\,dx\,ds\\
    &\hspace{3em} = \int_0^t\!\!\int_D 
      \gamma_\lambda(\nabla u_\lambda(s,x))(\nabla u_\lambda(s,x) 
      - \nabla(G \cdot W(s,x)))\,dx\,ds\\
    &\hspace{3em} \geq \int_0^t\!\!\int_D k_\lambda(\nabla u_\lambda(s,x))\,dx\,ds
      - \int_0^t\!\!\int_D k(\nabla(G\cdot W(s,x)))\,dx\,ds,
  \end{align*}
  hence also
  \begin{align*}
  &\frac12 \norm[\big]{u_\lambda(t)}^2 
  + 2 \int_0^t\!\!\int_D  {k_\lambda(\nabla u_\lambda(s, x))\,dx\,ds}
  + \lambda\int_0^t\norm[\big]{\nabla u_\lambda(s)}^2\,ds\\
  &\hspace{3em} \leq \norm[\big]{u_0}^2 + \norm[\big]{G\cdot W(t)}^2
  + \lambda \int_0^t \norm[\big]{\nabla(G\cdot W(s))}^2\,ds\\
  &\hspace{3em}\quad + 2\int_0^t\!\!\int_D k(\nabla(G\cdot W(s,x)))\,dx\,ds.
  \end{align*}
  Taking the supremum with respect to $t$ yields
  \begin{align*}
  &\norm[\big]{u_\lambda}^2_{C_t L^2_x}
  + \norm[\big]{k_\lambda(\nabla u_\lambda)}_{L^1_{t,x}}
  +\lambda \norm[\big]{\nabla u_\lambda}^2_{L^2_{t,x}}\\
  &\hspace{3em} \lesssim \norm[\big]{u_0}^2_{L^2_x}
  + \norm[\big]{G\cdot W}^2_{L^\infty_tL^2_x}
  + \norm[\big]{G\cdot W}^2_{L^2_t H^1_0}
  + \norm[\big]{k(\nabla (G\cdot W))}_{L^1_{t,x}}.
  \end{align*}
  As already observed above, the first three terms on the right-hand
  side are clearly finite. Moreover, since $V_0 \embed W^{1,\infty}$,
  one has
  \[ 
  \norm[\big]{k(\nabla(G\cdot W))}_{L^1_{t,x}} \lesssim_{T,D}
  \norm[\big]{k(\nabla(G\cdot W))}_{L^\infty_{t,x}}
  < \infty
  \]
  by the continuity of $k$. Since $\omega$ was chosen arbitrarily in
  $\Omega'$, the proof is completed.
\end{proof}

\begin{lemma}
  \label{lm:L1t}
  There exists a set $\Omega'$, with $\P(\Omega')=1$, such that, for
  all $\omega \in \Omega'$, the families $(\gamma_\lambda(\nabla
  u_\lambda))$ and $(\nabla u_\lambda)$ are relatively weakly compact
  in $L^1_{t,x}$.
\end{lemma}
\begin{proof}
  Let $\Omega'$ be defined as in the proof of Lemma \ref{lm:stp}, and
  fix an arbitrary $\omega \in \Omega'$. By \eqref{eq:pippo}, since
  $v_\lambda=u_\lambda - G \cdot W$, it follows that
  \begin{align*}
  &\int_0^t\!\!\int_D\gamma_\lambda(\nabla u_\lambda(s)) \cdot 
  \nabla u_\lambda(s)\,dx\,ds\\
  &\hspace{3em} \leq
  \frac12 \norm{u_0}^2 + \frac12 \norm{G\cdot W(t)}^2
  + \frac12 \int_0^t \norm{G\cdot W(s)}_{H^1_0}^2\,ds\\
  &\hspace{3em} \quad + \int_0^t\!\!\int_D
  \gamma_\lambda(\nabla u_\lambda(s))\cdot\nabla (G\cdot W(s))\,dx\,ds
  \end{align*}
  for all $t \leq T$. Thanks to Young's inequality, convexity of
  $k^*$, and $k^*(0)=0$, one has
  \begin{align*}
    \gamma_\lambda(\nabla u_\lambda)\cdot\nabla (G\cdot W) &=
    \frac12 \gamma_\lambda(\nabla u_\lambda) \cdot 2\nabla (G\cdot W)\\
    &\leq \frac12 k^*\bigl( \gamma_\lambda(\nabla u_\lambda)\bigr)
    + k(2\nabla(G \cdot W)).
  \end{align*}
  Recalling that
  $k^*(\gamma_\lambda(x)) \leq \gamma_\lambda(x) \cdot x$ for all
  $x \in \erre^n$, rearranging terms one gets
  \begin{align*}
  \int_0^T\!\!\int_D k^*(\nabla u_\lambda(s)) \,dx\,ds &\lesssim
    \norm{u_0}^2 + \norm{G\cdot W(T)}^2
  + \int_0^T \norm{G\cdot W(t)}_{H^1_0}^2\,ds\\
  &\quad + \int_0^T\!\!\int_D k\bigl(2\nabla (G\cdot W(s))\bigr) \,dx\,ds,
  \end{align*}
  where all terms on the right-hand side are finite, as already
  established in the proof of Lemma \ref{lm:stp}. Appealing again to
  the criteria by de la Vall\'ee Poussin and Dunford-Pettis, we
  immediately infer that
  $(\gamma_\lambda(\nabla u_\lambda(\omega,\cdot)))$ is relatively
  weakly compact in $L^1_{t,x}$.

  Denoting by $M$ (a constant depending on $\omega$) the right-hand
  side of the previous inequality, the above estimates also yield
  \[
  \norm[\big]{\gamma_\lambda(\nabla u_\lambda) \cdot \nabla u_\lambda}_{L^1_{t,x}}
  \lesssim M,
  \]
  hence also, recalling that
  $k((I+\lambda\gamma)^{-1}x) \leq \gamma_\lambda(x) \cdot x$,
  \[
  \norm[\big]{k\bigl((I+\lambda\gamma)^{-1}\nabla u_\lambda\bigr)}_{L^1_{t,x}}
  \lesssim M.
  \]
  This implies, in complete analogy to the previous case, that
  $\bigl((I+\lambda\gamma)^{-1}\nabla u_\lambda\bigr)$ is relatively weakly
  compact in $L^1_{t,x}$. Since
  \[
  \nabla u_\lambda = \lambda \gamma_\lambda(\nabla u_\lambda)
  + (I+\lambda\gamma)^{-1}\nabla u_\lambda,
  \]
  the relative weak compactness of $(\nabla u_\lambda(\omega,\cdot))$
  in $L^1_{t,x}$ follows immediately.
\end{proof}

\subsection{Proof of Proposition \ref{prop:V0}}
Let $\omega \in \Omega'$ be arbitrary but fixed, where $\Omega'$ is a
subset of $\Omega$ with probability one, chosen as in the proof of
Lemma~\ref{lm:stp}. The relative weak compactness of
$(\gamma_\lambda(\nabla u_\lambda))$ in $L^1_{t,x}$, proved in
Lemma~\ref{lm:L1t}, implies that there exists $\eta \in L^1_{t,x}$ such
that $\gamma_{\mu}(\nabla u_{\mu}) \to \eta$ weakly in $L^1_{t,x}$,
where $\mu$ is a subsequence of $\lambda$. This in turn implies that
\[
\int_0^t \div\gamma_{\mu}(\nabla u_{\mu}(s))\,ds \longrightarrow
\int_0^t \div\eta(s)\,ds
\qquad \text{weakly in } V_0'
\]
for all $t \in [0,T]$. In fact, for any $\phi_0 \in V_0$, setting
$\phi:=s \mapsto 1_{[0,t]}(s) \phi_0 \in L^\infty_t V_0$, recalling
that $V_0\embed W^{1,\infty}$, we have
\begin{align*}
&\int_0^t \ip[\big]{-\div\gamma_{\mu}(\nabla u_{\mu}(s))}{\phi_0}_{V_0}\,ds =
\int_0^T \ip[\big]{-\div\gamma_{\mu}(\nabla u_{\mu}(s))}{\phi(s)}_{V_0}\,ds\\
&\hspace{3em} = \int_0^T\!\!\int_D \gamma_{\mu}(\nabla u_{\mu}(s)) \cdot%
  \nabla\phi(s)\,ds\\
&\hspace{3em}\quad 
\longrightarrow \int_0^T\!\!\int_D \eta(s)\cdot\nabla \phi(s)\,ds
    = \int_0^t \ip[\big]{-\div\eta(s)}{\phi_0}\,ds
\end{align*}
as $\mu \to 0$.
Moreover, $\sqrt{\lambda} u_\lambda$ is bounded in $L^2_t H^1_0$ thanks
to Lemma~\ref{lm:stp}, hence, recalling that $\Delta$ is an
isomorphism of $H^1_0$ and $H^{-1}$, $\lambda \Delta u_\lambda \to 0$
in $L^2_t H^{-1}$ as $\lambda \to 0$, in particular
\[
\lambda \int_0^t \Delta u_\lambda(s)\,ds \longrightarrow 0
\qquad \text{in } H^{-1}
\]
for all $t \in [0,T]$ as $\lambda \to 0$.
Therefore, considering the regularized equation
\[
u_\mu(t) - \int_0^t \div \gamma_\mu(\nabla u_\mu(s))\,ds
- \mu \int_0^t \Delta u_\mu(s)\,ds = u_0 + G \cdot W(t)
\]
and passing to the limit as $\mu \to 0$, we infer that $u_\mu(t) \to
u(t)$ weakly in $V_0'$ for all $t \in [0,T]$, hence one can write
\begin{equation}
\label{eq:reglim}
u(t) - \int_0^t \div\eta(s)\,ds  = u_0 + G \cdot W(t)
\qquad \text{in $V_0'$}
\end{equation}
for all $t \in [0,T]$. Since $\div\eta \in L^1_t V_0'$ and
$G \cdot W \in L^\infty_t V_0$, it immediately follows that
$u \in C_t V_0'$.
Moreover, since, thanks to Lemma~\ref{lm:stp}, $(u_\mu(t))$ is bounded
in $L^2$, we also have $u_\mu(t) \to u(t)$ weakly in $L^2$. In fact,
let $\varepsilon>0$ and $\psi \in L^2$ be arbitrary. Since $V_0$ is
dense in $L^2$, there exists $\phi \in V_0$ with
$\norm[\big]{\psi-\phi}<\varepsilon$, and one can write
\[
\abs[\big]{\ip[\big]{u_\mu(t)-u_\nu(t)}{\psi}} \leq
\abs[\big]{\ip[\big]{u_\mu(t)-u_\nu(t)}{\psi-\phi}}
+ \abs[\big]{\ip[\big]{u_\mu(t)-u_\nu(t)}{\phi}},
\]
where the second term on the right-hand side converges to zero as
$\mu,\,\nu \to 0$, and
\[
\abs[\big]{\ip[\big]{u_\mu(t)-u_\nu(t)}{\psi-\phi}} \leq
\norm[\big]{u_\mu(t)-u_\nu(t)} \, \norm[\big]{\psi-\phi}
< N \varepsilon,
\]
so that, recalling that Hilbert spaces are weakly sequentially
complete, $u_\mu(t)$ converges weakly in $L^2$, necessarily to $u(t)$,
for all $t \in [0,T]$. This also immediately implies that $u \in
L^\infty_t L^2_x$. From this, together with $u \in
C_t V_0'$, it follows in turn that $u \in C_w([0,T];L^2)$ by a criterion
due to Strauss (see~\cite[Theorem~2.1]{Strauss} -- here and below
$C_w([0,T];E)$ stands for the space of space of weakly continuous
functions from $[0,T]$ to a Banach space $E$).
Furthermore, since all terms in \eqref{eq:reglim} except the second
one on the left-hand side take values in $L^2$, it follows that
\eqref{eq:reglim} is satisfied also as an identity in $L^2$.

Let us show that $u \in L^1_t W^{1,1}_0$: the relative weak
compactness of $(\nabla u_\lambda)$ in $L^1_{t,x}$, proved in
Lemma~\ref{lm:L1t}, implies that there exists $v \in L^1_{t,x}$ such
that, along a subsequence of $\lambda$ which can be assumed to
coincide with $\mu$, $\nabla u_{\mu} \to v$ weakly in
$L^1_{t,x}$. Taking into account that $u_{\mu} \in H^1_0$ for all
$\mu$ and that $u_\mu \to u$ weakly* in $L^\infty_tL^2_x$, it easily
follows that $v=\nabla u$ a.e. in $[0,T] \times D$ and that $u \in
L^1_t W^{1,1}_0$.

\medskip

As a next step, we are going to show that $\eta = \gamma(\nabla u)$
a.e. in $(0,T) \times D$. For this we shall need the ``energy''
identity proved in the following lemma.
\begin{lemma}
  \label{lm:testing}
  Assume that
  \[
  y(t) - \int_0^t \div \zeta(s)\,ds = y_0 + f(t)
  \qquad\text{in } L^2 \quad \forall t\in[0,T],
  \]
  where $y_0 \in L^2_x$, $y \in L^\infty_tL^2_x \cap L^1_tW^{1,1}_0$,
  $\zeta \in L^1_{t,x}$, and $f \in L^2_tV_0$ with $f(0)=0$.
  Furthermore, assume that there exists $c>0$ such that
  \[
  k(c\nabla y) + k^*(c\zeta) \in L^1_{t,x}.
  \]
  Then
  \[
  \norm[\big]{y(t) - f(t)}^2
  + 2\int_0^t\!\!\int_D \zeta(s,x) \cdot%
      \nabla\bigl(y(s,x)-f(s,x)\bigr)\,dx\,ds
  = \norm[\big]{y_0}^2 \qquad \forall t\in[0,T].
  \]
\end{lemma}
\begin{proof}
  The proof if analogous to that of Proposition~\ref{prop:ito}, of
  which we borrow the notation and the setup. In particular, let $m
  \in \mathbb{N}$ be such that
  \[
  y^\delta(t) - \int_0^t \div \zeta^\delta(s)\,ds = y^\delta_0 + f^\delta(t)
  \qquad\text{in } L^2 \quad \forall t\in[0,T],
  \]
  hence, by Lemma~\ref{lm:AC},
  \[
  \norm[\big]{y^\delta(t) - f^\delta(t)}^2
  + 2\int_0^t\!\!\int_D \zeta^\delta \cdot%
      \nabla\bigl( y^\delta - f^\delta \bigr)
  = \norm[\big]{y^\delta_0}^2 \qquad \forall t\in[0,T],
  \]
  where, as $\delta \to 0$,
  $\norm[\big]{y^\delta(t) - f^\delta(t)}^2 \to \norm[\big]{y(t) -
    f(t)}^2$ for all $t \in ]0,T]$ and
  $\norm[\big]{y^\delta_0}^2 \to \norm[\big]{y_0}^2$. Moreover, since
  $y^\delta - f^\delta \to y - f$ in $L^1_tW^{1,1}_0$ and
  $\zeta^\delta \to \zeta$ in $L^1_{t,x}$, we have that, up to
  selecting a subsequence,
  \[
  \zeta^\delta \cdot \nabla\bigl(y^\delta - f^\delta\bigr) \longrightarrow
  \zeta \cdot \nabla\bigl(y - f\bigr)
  \]
  almost everywhere in $[0,T] \times D$. Therefore, taking Vitali's
  theorem into account, the lemma is proved if we show that
  $\zeta^\delta \cdot \nabla(y^\delta - f^\delta)$ is uniformly
  integrable: one has, by Young's inequality and convexity,
  \begin{align*}
  \frac{c^2}2 \zeta^\delta \cdot \nabla(y^\delta - f^\delta) &\leq
  k\bigl( c/2 (\nabla y^\delta - \nabla f^\delta) \bigr) 
  + k^*\bigl( c\zeta^\delta \bigr)\\
  &\leq \frac12 k\bigl(c\nabla y^\delta \bigr)
  + \frac12 k\bigl(c(-\nabla f^\delta) \bigr)
  + k^*\bigl( c\zeta^\delta \bigr),
  \end{align*}
  as well as
  \begin{align*}
  - \frac{c^2}2 \zeta^\delta \cdot \nabla(y^\delta - f^\delta) &\leq
  k\bigl( c/2 (-\nabla y^\delta + \nabla f^\delta) \bigr) 
  + k^*\bigl( c\zeta^\delta \bigr)\\
  &\leq \frac12 k\bigl(c(-\nabla y^\delta) \bigr)
  + \frac12 k\bigl(c\nabla f^\delta \bigr)
  + k^*\bigl( c\zeta^\delta \bigr),
  \end{align*}
  hence
  \begin{align*}
  c^2 \abs[\big]{\zeta^\delta \cdot \nabla(y^\delta - f^\delta)}
  &\leq k\bigl(c\nabla y^\delta \bigr) + k\bigl(c(-\nabla y^\delta) \bigr)\\
  &\quad + k\bigl(c\nabla f^\delta \bigr)
  + k\bigl(c(-\nabla f^\delta) \bigr) + 4k^*\bigl( c\zeta^\delta \bigr).
  \end{align*}
  It follows by Jensen's inequality for sub-Markovian operators,
  recalling that $(I-\delta\Delta)^{-m}$ and $\nabla$ commute, that
  \begin{align*}
  c^2 \abs[\big]{\zeta^\delta \cdot \nabla(y^\delta - f^\delta)}
  &\leq (I-\delta\Delta)^{-m} \Bigl(%
   k(c\nabla y) + k\bigl(c(-\nabla y) \bigr)\\
  &\qquad + k(c\nabla f)
  + k\bigl(c(-\nabla f) \bigr) + 4k^*(c\zeta) \Bigr),
  \end{align*}
  where $k(c\nabla y)$ and $k^*(c\zeta)$ belong to $L^1_{t,x}$ by
  assumption, and the same holds for $k(c\nabla f) + k(c(-\nabla f))$
  because $f \in W^{1,\infty}$. Moreover, the hypothesis
  $\limsup_{\abs{x} \to \infty} k(-x)/k(x)<\infty$ implies that
  \[
  \int_0^T\!\!\int_D k(c(-\nabla y))
  \lesssim 1 + \int_0^T\!\!\int_D k(\nabla y) < \infty,
  \]
  therefore, taking into account that $(I-\delta\Delta)^{-m}$ is a
  contraction in $L^1$, we obtain that
  $c^2\abs{\zeta^\delta \cdot \nabla(y^\delta - f^\delta)}$ is
  dominated by a sequence that converges in $L^1_{t,x}$, which
  immediately implies that
  $\zeta^\delta \cdot \nabla(y^\delta - f^\delta)$ is uniformly
  integrable in $[0,T] \times D$.
\end{proof}

As in the proof of Lemma~\ref{lm:stp}, it follows from \eqref{eq:vul}
and Lemma~\ref{lm:AC} that
\begin{align*}
  &\frac12\norm{v_\lambda(t)}^2 + \int_0^t\!\!\int_D
  \gamma_\lambda(\nabla u_\lambda(s)) \cdot \nabla v_\lambda(s)\,dx\,ds\\
  &\hspace{3em}
  + \lambda\int_0^t\!\!\int_D \nabla u_\lambda(s) \cdot
    \nabla v_\lambda(s)\,dx\,ds
  = \frac12\norm{u_0}^2
\end{align*}
for all $t \in [0,T]$, where $v_\lambda = u_\lambda - G \cdot W$. This
immediately implies
\begin{equation}
  \label{eq:czt}
  \begin{split}
  &\frac12\norm{v_\lambda(t)}^2 + \int_0^t\!\!\int_D
  \gamma_\lambda(\nabla u_\lambda(s)) \cdot \nabla u_\lambda(s)\,dx\,ds\\
  &\hspace{3em}
  \leq \frac12\norm{u_0}^2
  + \int_0^t\!\!\int_D
  \gamma_\lambda(\nabla u_\lambda(s)) \cdot \nabla (G \cdot W(s))\,dx\,ds\\
  &\hspace{3em} \quad + \lambda \int_0^t\!\!\int_D \nabla u_\lambda(s) \cdot%
  \nabla(G \cdot W(s))\,dx\,ds,
  \end{split}
\end{equation}
where
\[
\liminf_{\mu \to 0} \norm[\big]{v_\mu(t)} \geq
\norm[\big]{u(t) - G \cdot W(t)} \qquad \forall t \in [0,T]
\]
by the weak lower semicontinuity of the norm and the weak convergence
of $u_\mu(t)$ to $u(t)$ in $L^2$. Moreover, recalling that
$\gamma_\mu(\nabla u_\mu) \to \eta$ weakly in $L^1_{t,x}$ and
$\nabla(G\cdot W) \in L^\infty_{t,x}$, as $V_0 \embed
W^{1,\infty}$, we have
\[
\int_0^t\!\!\int_D
\gamma_\mu(\nabla u_\mu(s)) \cdot \nabla (G \cdot W(s))\,dx\,ds
\longrightarrow
\int_0^t\!\!\int_D \eta(s) \cdot \nabla (G \cdot W(s))\,dx\,ds.
\]
The last term on the right-hand side of \eqref{eq:czt} converges to
zero as $\mu \to 0$ because $(\nabla u_\mu)$ is bounded in
$L^1_{t,x}$ and $\nabla(G\cdot W) \in L^\infty_{t,x}$. We have thus
obtained
\begin{align*}
&\limsup_{\mu \to 0} \int_0^T\!\!\int_D
  \gamma_\mu(\nabla u_\mu(s)) \cdot \nabla u_\mu(s)\,dx\,ds\\
&\hspace{3em} \leq \frac12 \norm[\big]{u_0}^2 
- \frac12 \norm[\big]{u(T) - G \cdot W(T)}^2
+ \int_0^t\!\!\int_D \eta(s) \cdot \nabla (G \cdot W(s))\,dx\,ds.
\end{align*}
By Lemma~\ref{lm:testing} we have
\begin{align*}
  &\frac12 \norm[\big]{u_0}^2 
- \frac12 \norm[\big]{u(T) - G \cdot W(T)}^2
+ \int_0^T\!\!\int_D \eta(s) \cdot \nabla (G \cdot W(s))\,dx\,ds\\
&\hspace{3em} = \int_0^T\!\!\int_D \eta(s) \cdot \nabla u(s)\,dx\,ds,
\end{align*}
which implies that
\[
\limsup_{\mu \to 0} \int_0^T\!\!\int_D 
  \gamma_{\mu}(\nabla u_{\mu}) \cdot \nabla u_{\mu} \,dx\,ds
  \leq \int_0^T\!\!\int_D \eta\cdot\nabla u \,dx\,ds.
\]
Moreover, since
\[
\gamma_\mu(x)\cdot (I+\mu\gamma)^{-1}x=
\gamma_\mu(x) \cdot x - \mu\abs{\gamma_\mu(x)}^2 \leq 
\gamma_\mu(x)\cdot x
\]
for all $x \in \erre^n$, we obtain
\[
\limsup_{\mu \to 0} \int_0^T\!\!\int_D \gamma_{\mu}(\nabla u_\mu)
\cdot (I+\mu\gamma)^{-1} \nabla u_{\mu} \,dx\,ds
\leq \int_0^T\!\!\int_D \eta \cdot \nabla u\,dx\,ds,
\]
where $(I+\mu\gamma)^{-1}\nabla u_{\mu} \to \nabla u$ and
$\gamma_{\mu}(\nabla u_{\mu}) \to \eta$ weakly in $L^1_{t,x}$.  In
particular, the weak lower semicontinuity of convex integrals yields
\begin{align*}
&\int_0^T\!\!\int_D \bigl( k(\nabla u)+k^*(\eta) \bigr)\\
&\hspace{3em} \leq 
\liminf_{\mu \to 0} \int_0^T\!\!\int_D \bigl( k((I+\mu \gamma)^{-1} \nabla u_\mu)
   + k^*(\gamma_\mu(\nabla u_\mu)) \bigr)\,dx\,dt\\
&\hspace{3em} = \liminf_{\mu \to 0} \int_0^T\!\!\int_D
\gamma_\mu(\nabla u_\mu)\cdot(I+\mu \gamma)^{-1} \nabla u_\mu\,dx\,dt < N,
\end{align*}
where $N=N(\omega)$ is a constant.
Recalling that $\gamma_{\mu} \in \gamma((I+\mu\gamma)^{-1})$ and
$\gamma=\partial k$, we have
\[
  k((I+\mu\gamma)^{-1}\nabla u_{\mu})+
  \gamma_{\mu}(\nabla
  u_{\mu})\cdot(z-(I+\mu\gamma)^{-1}\nabla u_{\mu})
  \leq k(z) \qquad \forall z\in\erre^n.
\]
From this it follows, again by the weak lower semicontinuity of convex
integrals, that
\[
\int_0^T\!\!\int_D k(\nabla u)
+ \int_0^T\!\!\int_D \eta \cdot (\zeta - \nabla u)
\leq \int_0^T\!\!\int_D k(\zeta) 
\qquad \forall \zeta \in L^\infty_{t,x}.
\]
Let $A$ be an arbitrary Borel subset of $(0,T) \times D$,
$z_0 \in \erre^n$, $R>0$ a constant, and
\[
\zeta_R := z_0 1_{A} + T_R(\nabla u) 1_{A^c},
\]
where $T_R:\erre^n \to \erre^n$, is the truncation operator
\[
T_R: x \longmapsto
\begin{cases}
x, & \abs{x} \leq R,\\[4pt]
\displaystyle R x/\abs{x}, & \abs{x} > R.
\end{cases}
\]
Then $\zeta_R \in L^\infty_{t,x}$, and
\begin{align*}
&\int_A k(\nabla u) + \int_A \eta \cdot (z_0 - \nabla u)
\leq \int_A k(z_0)\\
&\hspace{3em} + \int_{A^c} \bigl( k(T_R(\nabla u)) - k(\nabla u) \bigr)
+ \int_{A^c} \eta \cdot \bigl( T_R(\nabla u) - \nabla u \bigr),
\end{align*}
where $T_R(\nabla u) \to \nabla u$ and $k(T_R(\nabla u)) \to k(\nabla u)$
a.e. in $(0,T) \times D$ as $R \to \infty$, as well as
\[
\abs[\big]{T_R(\nabla u) - \nabla u} \leq 2 \abs[\big]{\nabla u},
\qquad \abs[\big]{k(T_R(\nabla u)) - k(\nabla u)}
\lesssim 1 + k(\nabla u)
\]
(the latter inequality follows by the assumptions on the behavior of
$k$ at infinity).  Since $k(\nabla u)$, $k^*(\eta)
\in L^1_{t,x}$, the dominated convergence theorem implies that
\[
\int_A k(\nabla u) + \int_A \eta \cdot (z_0 - \nabla u)
\leq \int_A k(z_0)
\]
for arbitrary $z_0$ and $A$, hence also that
\[
k(\nabla u) + \eta \cdot (z_0 - \nabla u)
\leq k(z_0)
\]
a.e. in $(0,T) \times D$ for all $z_0 \in \erre^n$. By definition of
subdifferential it follows immediately that
$\eta = \gamma(\nabla u)$ a.e. in $(0,T) \times D$.

\medskip

Let us now show, still keeping $\omega$ fixed, that the limit $u$
constructed above is unique. In particular, since
$\eta = \gamma(\nabla u)$, it is also unique. Assume that there exist
$u_1$, $u_2$ such that
\[
u_i(t) - \int_0^t \div \gamma(\nabla u_i(s))\,ds  = u_0 + G
\cdot W(t), \qquad i=1,2,
\]
in $L^2$ for all $t\in [0,T]$. Setting $v = u_1-u_2$ and
$\zeta = \gamma(\nabla u_1) - \gamma(\nabla u_2)$, it is enough to
show that
\[
v(t) - \int_0^t \div \zeta(s)\,ds  = 0
\]
in $L^2$ for all $t\in [0,T]$ implies $v=0$. To this aim, it suffices to
note that, by Lemma \ref{lm:testing},
\[
\frac12 \norm[\big]{v(t)}^2
+ \int_0^t\!\!\int_D \zeta \cdot \nabla v = 0
\]
for all $t \in [0,T]$. The monotonicity of $\gamma$ immediately
implies $v=0$, i.e. $u_1=u_2$, so that uniqueness of $u$ is proved.

\medskip

The process $u$ has been constructed for each $\omega$ in a set of
probability one via limiting procedures along sequences that depend on
$\omega$ itself. Of course such a construction, in general, does not
produce a measurable process. In our situation, however, uniqueness of
$u$ allows us to even prove that $u$ is predictable. The following
simple observation is crucial: we have proved that from any
subsequence of $\lambda$ one can extract a further subsequence $\mu$,
depending on $\omega$, such that $u_\mu$ converges to $u$ as
$\mu \to 0$, in several topologies, and that the limit $u$ is unique.
This implies, by a classical criterion, that the same convergences
hold along the original sequence $\lambda$, which does \emph{not}
depend on $\omega$. In particular,
$u_\lambda(\omega,t) \to u(\omega,t)$ weakly in $L^2$ for all
$t\in[0,T]$ and for $\P$-a.s. $\omega$. Let us show that
$u_\lambda \to u$ weakly in $\bL^1 L^1_t L^2_x$: for an arbitrary
$\phi \in \bL^\infty L^\infty_t L^2_x$, we have
\[
\ip[\big]{u_\lambda(\omega,t)}{\phi(\omega,t)}
\longrightarrow \ip[\big]{u(\omega,t)}{\phi(\omega,t)}
\]
a.e. in $\Omega \times [0,T]$, as well as
\begin{align*}
\E\int_0^T \ip[\big]{u_\lambda(\omega,t)}{\phi(\omega,t)}^2\,dt
&\leq \E\int_0^T \norm[\big]{u_\lambda(\omega,t)}^2
      \norm[\big]{\phi(\omega,t)}^2\,dt\\
&\leq \norm[\big]{\phi}^2_{\bL^\infty L^\infty_t L^2_x}
      \E\int_0^T \norm[\big]{u_\lambda(\omega,t)}^2\,dt < N
\end{align*}
for a constant $N$ independent of $\lambda$, because $(u_\lambda)$ is
bounded in $\bL^2 L^2_{t,x}$ by Lemma~\ref{lm:aspetta}. Then
$\ip{u_\lambda}{\phi}$ is uniformly integrable in
$\Omega \times [0,T]$ by the criterion of de la Vall\'ee Poussin,
hence $\ip{u_\lambda}{\phi} \to \ip{u}{\phi}$ in $\bL^1 L^1_t$ by
Vitali's theorem. Since $\phi \in \bL^\infty L^\infty_t L^2_x$ is
arbitrary, it follows that $u_\lambda \to u$ weakly in
$\bL^1 L^1_t L^2_x$. Mazur's lemma (see, e.g., \cite[p.~360]{Bbk:EVT})
implies that there exists a sequence $(\zeta_n)$ of convex
combinations of $u_\lambda$ such that
$\zeta_n(\omega,t) \to u(\omega,t)$ in $L^2$ in
$\P \otimes dt$-measure, hence a.e. in $\Omega \times [0,T]$ along a
subsequence. Since $(u_\lambda)$ is a collection of $L^2$-valued
predictable processes, the same holds for $(\zeta_n)$, so that the
$\P \otimes dt$-a.e. pointwise limit $u$ of (a subsequence of)
$\zeta_n$ is an $L^2$-valued predictable process as well.
We also have that $u \in \bL^2 L^\infty_t L^2_x$, as it follows by
$u_\lambda \to u$ in $\bL^1 L^1_t L^2_x$ and the boundedness of
$(u_\lambda)$ in $\bL^2 L^\infty_t L^2_x$.

Moreover, recalling that $\nabla u_\lambda \to \nabla u$ and
$\gamma_\lambda(\nabla u_\lambda) \to \eta$ weakly in $L^1_{t,x}$
$\P$-a.s., and that, by Lemma~\ref{lm:L1a}, $(\nabla u_\lambda)$ and
$(\gamma_\lambda(\nabla u_\lambda))$ are bounded in $\bL^1 L^1_{t,x}$,
an entirely analogous argument shows that $\nabla u_\lambda \to \nabla
u$ and $\gamma_\lambda(\nabla u_\lambda) \to \eta=\gamma(\nabla u)$
weakly in $\bL^1 L^1_{t,x}$. This implies that $\eta$ is a measurable
adapted process, as well as, by weak lower 
semicontinuity of the norm,
\begin{align*}
  \E \norm[\big]{\nabla u}_{L^1_{t,x}} &\leq \liminf_{\lambda \to 0}
  \E \norm[\big]{\nabla u_{\lambda}}_{L^1_{t,x}} < \infty,\\
  \E \norm[\big]{\eta}_{L^1_{t,x}} &\leq \liminf_{\lambda \to 0}
  \E \norm[\big]{\gamma_\lambda(\nabla u_\lambda)}_{L^1_{t,x}} < \infty.
\end{align*}
We can hence conclude that
\begin{align*}
  &u \in \bL^2 L^\infty_t L^2_x \,\cap\, \bL^1 L^1_t W^{1,1}_0,\\
  &\eta = \gamma(\nabla u) \in \bL^1 L^1_{t,x}.
\end{align*}
Finally, Lemma~\ref{lm:aspetta} and \eqref{eq:pluto} yield
\begin{align*}
  &\E\int_0^T\!\!\int_D \bigl( k((I+\lambda\gamma)^{-1}\nabla u_\lambda)%
   + k^*(\gamma_\lambda(\nabla u_\lambda)) \bigr)\\
  &\hspace{3em} <  N \Bigl( \E\norm[\big]{u_0}^2
   + \E\int_0^T \norm[\big]{G(s)}^2_{\cL(H,L^2)}\,ds \Bigr),
\end{align*}
where, by the weak lower semicontinuity of convex integrals and
$(I+\lambda\gamma)^{-1}\nabla u_\lambda \to \nabla u$,
$\gamma_\lambda(\nabla u_\lambda) \to \eta$ weakly in $L^1_{t,x}$
$\P$-a.s., one has
\[
\int_0^T\!\!\int_D \bigl( k(\nabla u) + k^*(\eta) \bigr) \leq
\liminf_{\lambda \to 0} 
\int_0^T\!\!\int_D \bigl( k((I+\lambda\gamma)^{-1}\nabla u_\lambda)%
   + k^*(\gamma_\lambda(\nabla u_\lambda))\bigr)
\]
$\P$-a.s., hence, by Fatou's lemma,
\begin{equation}
\label{eq:ona}
\begin{split}
  \E\int_0^T\!\!\int_D \bigl( k(\nabla u) + k^*(\eta) \bigr)
  &\leq \liminf_{\lambda \to 0} \E
  \int_0^T\!\!\int_D \bigl(k((I+\lambda\gamma)^{-1}\nabla u_\lambda)%
   + k^*(\gamma_\lambda(\nabla u_\lambda))\bigr)\\
  &<  N \Bigl( \E\norm[\big]{u_0}^2
   + \E\int_0^T \norm[\big]{G(s)}^2_{\cL(H,L^2)}\,ds \Bigr) < \infty.
\end{split}
\end{equation}

\begin{rmk}
  The proof of uniqueness of $u$ does \emph{not} depend on $\gamma$
  being single-valued. In particular, all results on $u$ obtained thus
  far, including the predictability of $u$, can be obtained under the
  more general assumption that $\gamma$ is an everywhere defined
  maximal monotone graph on $\erre^n \times \erre^n$, with
  $\gamma=\partial k$. However, in this more general framework, the
  uniqueness of $\eta$ does \emph{not} follow, because the divergence
  is not injective. This implies that we would not be able even to
  prove that $\eta$ is a measurable process (with respect to the
  product $\sigma$-algebra of $\mathscr{F}$ and the Borel
  $\sigma$-algebra of $[0,T]$).
\end{rmk}


\ifbozza\newpage\else\fi
\section{Well-posedness with additive noise}
\label{sec:add}
We are now going to prove well-posedness for the equation
\begin{equation}
\label{eq:add}
du(t) - \div\gamma(\nabla u(t))\,dt = G(t)\,dW(t),
\qquad u(0)=u_0,
\end{equation}
where $G$ is no longer supposed to take values in $\cL^2(H,V_0)$, as
in the previous section, but just in $\cL^2(H,L^2)$. In other words,
we are considering equation \eqref{eq:0} with additive noise.

\begin{prop}  \label{prop:add}
  Assume that $u_0 \in \bL^2 L^2_x$ is $\cF_0$-measurable and that
  $G: \Omega \times [0,T] \to \cL^2(H,L^2)$ is measurable and adapted.
  Then equation \eqref{eq:V0} is well posed in $\mathscr{K}$.
  Moreover, its solution is pathwise weakly continuous with values in
  $L^2$.
\end{prop}
\begin{proof}
  Since one has $(I-\varepsilon\Delta)^{-m}: L^2 \to H^{2m} \cap
  H^1_0$ for any $m \in \mathbb{N}$, choosing $m > 1/2 +
  n/4$, the Sobolev embedding theorem yields $H^{2m} \embed
  W^{1,\infty}$, hence $V_0 := H^{2m} \cap H^1_0$ satisfies all
  hypotheses stated at the beginning of the previous section. In
  particular, setting
  \[
  G^\varepsilon := (I - \varepsilon \Delta)^{-m} G,
  \]
  the ideal property of Hilbert-Schmidt operators implies that
  $G^\varepsilon$ is an $\cL^2(H,V_0)$-valued measurable and adapted
  process such that
  \[
  \E\int_0^T \norm[\big]{G^\varepsilon(s)}^2_{\cL^2(H,V_0)}\,ds \lesssim
  \E\int_0^T \norm[\big]{G(s)}^2_{\cL^2(H,L^2)}\,ds < \infty.
  \]
  It follows by Proposition \ref{prop:V0} that, for any
  $\varepsilon>0$, there exists a unique predictable process
  \[
  u^\varepsilon \in \bL^2 L^\infty_t L^2_x \,\cap\, \bL^1 L^1_t W^{1,1}_0
  \]
  such that
  \begin{gather*}
    \eta^\varepsilon = \gamma(u^\varepsilon) \in \bL^1 L^1_{t,x},\\
    k(\nabla u^\varepsilon) + k^*(\eta^\varepsilon) \in L^1_{t,x}
    \quad \P\text{-a.s.},\\
    u^\varepsilon \in C_w([0,T];L^2) \quad \P\text{-a.s.},
  \end{gather*}
  satisfying
  \begin{equation}
    \label{eq:eps}
    u^\varepsilon(t) - \int_0^t \div\eta^\varepsilon(s)\,ds
    = u_0 + \int_0^t G^\varepsilon(s)\,dW(s)
  \end{equation}
  in $L^2$ for all $t \in [0,T]$. 

  In complete analogy to the previous section, the equation in
  $H^{-1}$
  \[
  u^\varepsilon_\lambda(t) 
  - \int_0^t \div\gamma_\lambda(\nabla u^\varepsilon_\lambda(s))\,ds 
  - \lambda \int_0 \Delta u^\varepsilon_\lambda(s)\,ds 
  = u_0 + \int_0^t G^\varepsilon(s)\,dW(s)
  \]
  admits a unique (variational) strong solution
  $u_\lambda^\varepsilon$ for any $\varepsilon>0$ and
  $\lambda>0$. Taking into account the monotonicity of
  $\gamma_\lambda$, It\^o's formula yields, for any $\delta>0$,
  \begin{align*}
  &\norm[\big]{u^\varepsilon_\lambda(t) - u^\delta_\lambda(t)}^2
  + \lambda \int_0^t
  \norm[\big]{\nabla(u^\varepsilon_\lambda(s) - u^\delta_\lambda(s))}^2\,ds\\
  &\hspace{3em} \lesssim \int_0^t
  \bigl( u^\varepsilon_\lambda(s) - u^\delta_\lambda(s) \bigr)
  \bigl( G^\varepsilon(s) - G^\delta(s) \bigr)\,dW(s)
  + \int_0^t \norm[\big]{G^\varepsilon(s) - G^\delta(s)}^2_{\cL^2(H,L^2)}\,ds.
  \end{align*}
  Taking supremum in time and expectation, it easily follows from
  Lemma \ref{lm:DY} that
  \[
  \E\sup_{t\leq T} \norm[\big]{u^\varepsilon_\lambda(t) - u^\delta_\lambda(t)}^2
  \lesssim \E\int_0^T
    \norm[\big]{G^\varepsilon(t) - G^\delta(t)}^2_{\cL^2(H,L^2)}.
  \]
  For arbitrary fixed $\varepsilon$, $\delta>0$, the proof of
  Proposition \ref{prop:V0} shows that
  \begin{align*}
    u^\varepsilon_\lambda &\longrightarrow u^\varepsilon
    & &\text{weakly* in } L^\infty_tL^2_x,\\
    \nabla u^\varepsilon_\lambda &\longrightarrow \nabla u^\varepsilon
    & &\text{weakly in } L^1_{t,x},\\
    \gamma_\lambda(\nabla u^\varepsilon_\lambda) 
    &\longrightarrow \eta^\varepsilon
    & &\text{weakly in } L^1_{t,x}
  \end{align*}
  $\P$-a.s. as $\lambda \to 0$, and the same holds replacing
  $\varepsilon$ with $\delta$. In particular, on a set of probability
  one,
  $u^\varepsilon_\lambda-u^\delta_\lambda \to u^\varepsilon-u^\delta$
  weakly* in $L^\infty_tL^2_x$ as $\lambda \to 0$, hence the weak* lower
  semicontinuity of the norm and Fatou's lemma imply
  \begin{align*}
  \E\norm[\big]{u^\varepsilon-u^\delta}^2_{L^\infty_tL^2_x} &\leq
  \liminf_{\lambda \to 0} \E
  \norm[\big]{u_\lambda^\varepsilon-u_\lambda^\delta}^2_{L^\infty_tL^2_x}\\
  &\lesssim \E\int_0^T
  \norm[\big]{G^\varepsilon(s) - G^\delta(s)}^2_{\cL^2(H,L^2))}\,ds.
  \end{align*}
  It follows by the ideal property of Hilbert-Schmidt operators, the
  contractivity of $(I-\varepsilon\Delta)^{-m}$, and the dominated
  convergence theorem, that
  \[
  \E\int_0^T \norm[\big]{G^\varepsilon(s) - G(s)}^2_{\cL^2(H,L^2))}\,ds 
  \longrightarrow 0
  \]
  as $\varepsilon \to 0$. This implies that $(u^\varepsilon)$ is a
  Cauchy sequence in $\bL^2 L^\infty_t L^2_x$, hence there exists a
  predictable $L^2$-valued process $u$ such that $u^\varepsilon$
  converges (strongly) to $u$ in $\bL^2 L^\infty_t L^2_x$ as
  $\varepsilon \to 0$.
  Moreover, by \eqref{eq:ona} there exists a constant $N$, independent
  of $\varepsilon$, such that
  \begin{equation}
    \label{eq:ekk}    
  \begin{split}
  &\E\int_0^T\!\!\int_D \bigl(
    k(\nabla u^\varepsilon) + k^*(\eta^\varepsilon)\bigr)\,dx\,ds\\
  &\hspace{3em} < N \Bigl( \E\norm[\big]{u_0}^2
   + \E\int_0^T \norm[\big]{G^\varepsilon(s)}^2_{\cL(H,L^2)}\,ds \Bigr)\\
  &\hspace{3em} \leq N \Bigl( \E\norm[\big]{u_0}^2
   + \E\int_0^T \norm[\big]{G(s)}^2_{\cL(H,L^2)}\,ds \Bigr),
  \end{split}
  \end{equation}
  where we have used again the ideal property of Hilbert-Schmidt
  operators and the contractivity of $(I-\varepsilon\Delta)^{-m}$ in
  the last step. The sequences $(\nabla u^\varepsilon)$ and
  $(\gamma(\nabla u^\varepsilon))$ are hence uniformly integrable on
  $\Omega \times [0,T] \times D$ by the criterion of de la Vall\'ee
  Poussin, hence relatively weakly compact in $\bL^1(L^1_{t,x})$ by the
  Dunford-Pettis theorem.  Therefore, passing to a subsequence of
  $\varepsilon$, denoted by the same symbol, there exist $v$ and
  $\eta$ such that $\nabla u^\varepsilon \to v$ and
  $\gamma(\nabla u^\varepsilon) \to \eta$ weakly in $\bL^1 L^1_{t,x}$
  as $\varepsilon \to 0$.  It is then straightforward to check
  that $v=\nabla u$ and
  \[
  u \in \bL^1 L^1_t W^{1,1}_0.
  \]
  An argument based on Mazur's lemma, entirely analogous to the one
  used in the proof of Proposition~\ref{prop:V0}, shows that $\eta$ is
  an $L^1$-valued adapted process.

  We can now pass to the limit as $\varepsilon \to 0$ in
  \eqref{eq:eps}. The strong convergence of $u^\varepsilon$ to $u$ in
  $\bL^2 L^\infty_t L^2_x$ implies that
  \[
  \operatorname*{ess\,sup}_{t\in[0,T]} \norm[\big]{u^\varepsilon(t)-u(t)} 
  \to 0
  \]
  in probability as $\varepsilon \to 0$. Let $\phi_0 \in V_0$ be
  arbitrary. Since $V_0 \embed L^\infty$, one has
  \[
  \ip[\big]{u^\varepsilon(t)}{\phi_0} \to \ip[\big]{u(t)}{\phi_0}
  \]
  in probability for almost all $t \in [0,T]$. Let us set, for an
  arbitrary but fixed $t \in [0,T]$,
  $\phi:s \mapsto 1_{[0,t]}(s) \phi_0 \in L^\infty_t V_0$.  Recalling
  that $\eta^\varepsilon=\gamma(\nabla u^\varepsilon) \to \eta$ weakly
  in $\bL^1 L^1_{t,x}$, it follows immediately that
  \begin{align*}
    - \int_0^t \ip{\div\eta^\varepsilon}{\phi_0}\,ds &=
    \int_0^T\!\!\int_D\eta^\varepsilon(s)\cdot\phi(s)\,ds \\
    &\quad \to \int_0^T\!\!\int_D\eta(s)\cdot\nabla\phi(s)\,ds
    = - \int_0^t \ip{\div\eta(s)}{\phi_0}\,ds
  \end{align*}
  weakly in $\bL^1$ as $\varepsilon \to 0$.
  Doob's maximal inequality and the convergence
  \[
  \E\int_0^T \norm[\big]{G^\varepsilon(t) - G(t)}_{\cL^2(H,L^2)}
  \longrightarrow 0
  \]
  as $\varepsilon \to 0$ readily yield also that
  $G^\varepsilon \cdot W(t) \to G \cdot W(t)$ in $L^2$ in probability
  for all $t \in [0,T]$. In particular, since $\phi_0\in V_0$ and
  $t \in [0,T]$ are arbitrary, we infer that
  \[
  u(t) - \int_0^t \div\eta(s)\,ds = 
  u_0 + \int_0^t B(s)\,dW(s)
  \]
  holds in $V_0'$ for almost all $t$. Recalling that
  $\eta \in L^1_{t,x}$, which implies in turn that
  $\div\eta\in L^1_t V_0'$, it follows that all terms except the first
  on the left-hand side have trajectories in $C_t V_0'$, hence
  that the identity holds for all $t \in [0,T]$. Moreover, thanks to
  Strauss' weak continuity criterion, $u \in C_t V_0'$ and
  $u \in L^\infty_t L^2_x$ imply $u \in C_w([0,T];L^2)$. Note also that
  all terms bar the second one on the left-hand side are $L^2$-valued,
  hence the identity holds also in $L^2$ for all $t \in [0,T]$.

  The weak convergences $\nabla u^\varepsilon \to \nabla u$ and
  $\eta^\varepsilon\to\eta$ in $\bL^1L^1_{t,x}$ and the weak lower
  semicontinuity of convex integrals yield, taking \eqref{eq:ekk} into
  account,
  \[
  \E\int_0^T\!\!\int_D \bigl(k(\nabla u)+k^*(\eta)\bigr) 
  < N \Bigl( \E\norm[\big]{u_0}^2
  + \E\int_0^T \norm[\big]{G(s)}^2_{\cL^2(H,L^2)}\,ds \Bigr).
  \]

  To complete the proof of existence, we only need to show that $\eta
  = \gamma(\nabla u)$ a.e.~in $\Omega \times (0,T) \times D$. Note
  that, by Proposition~\ref{prop:ito}, we have
  \begin{align*}
  &\frac{1}{2}\norm[\big]{u^\varepsilon(T)}^2
  + \int_0^T\!\!\int_D \eta^\varepsilon \cdot \nabla u^\varepsilon\\
  &\hspace{3em} = \frac12\norm[\big]{u_0}^2
  +\frac12\int_0^T\norm[\big]{G^\varepsilon(s)}^2_{\cL^2(H,L^2)}\,ds+
  \int_0^T u^\varepsilon(s)G^\varepsilon(s)\,dW(s)
  \end{align*}
  and
  \begin{align*}
  &\frac{1}{2}\norm[\big]{u(T)}^2
  + \int_0^T\!\!\int_D\eta\cdot\nabla u\\
  &\hspace{3em} = \frac12\norm[\big]{u_0}^2
  +\frac12\int_0^T\norm[\big]{G(s)}^2_{\cL^2(H,L^2)}\,ds+
  \int_0^Tu(s)G(s)\,dW(s),
  \end{align*}
  where, as $\varepsilon \to 0$, $\norm{u^\varepsilon(T)} \to
  \norm{u(T)}$ in $\bL^2$, thanks to the strong
  convergence of $u^\varepsilon$ to $u$ in $\bL^2L^\infty_tL^2_x$;
  \[
  \int_0^T\norm[\big]{G^\varepsilon(s)}^2_{\cL^2(H,L^2)}\,ds \longrightarrow
  \int_0^T\norm[\big]{G(s)}^2_{\cL^2(H,L^2)}\,ds
  \]
  in $\bL^2$ by an (already seen) argument involving the ideal
  property of Hilbert-Schmidt operators;
  \[
  \int_0^T u^\varepsilon(s)G^\varepsilon(s)\,dW(s) 
  \longrightarrow \int_0^Tu(s)G(s)\,dW(s)
  \]
  in $\bL^1$ as it follows by Lemma~\ref{lm:DY}. In particular, we infer
  \begin{align*}
  &\limsup_{\varepsilon\to0} \int_0^T\!\!\int_D
  \gamma(\nabla u^\varepsilon) \cdot \nabla u^\varepsilon \\
  &\hspace{3em} \leq \frac12\norm[\big]{u_0}^2-\frac{1}{2}\norm[\big]{u(T)}^2
  +\frac12\int_0^t\norm[\big]{G(s)}^2_{\cL^2(H,L^2)}\,ds+
  \int_0^tu(s)G(s)\,dW(s) \\
  &\hspace{3em} =\int_0^t\!\!\int_D\eta\cdot\nabla u,
  \end{align*}
  hence also, by Fatou's lemma,
  \[
  \limsup_{\varepsilon \to 0} \E\int_0^T\!\!\int_D
  \gamma(\nabla u^\varepsilon) \cdot \nabla u^\varepsilon
  \leq \E\int_0^t\!\!\int_D \eta \cdot \nabla u.
  \]
  Since $\nabla u^\varepsilon\to\nabla u$ and $\gamma(\nabla
  u^\varepsilon) \to \eta$ weakly in $\bL^1L^1_{t,x}$, recalling that
  $\gamma$ is maximal monotone, it follows that $\eta\in\gamma(\nabla
  u)$ a.e.~in $\Omega\times(0,T)\times D$ (see, e.g., \cite[Lemma~2.3,
  p.~38]{Barbu:type}).

  Let $u_{01}$, $u_{02} \in \bL^2L^2_x$ be $\mathscr{F}_0$-measurable,
  and $G_1$, $G_2:\Omega \times [0,T] \to \cL^2(H,L^2)$ be measurable
  adapted processes such that
  \[
  \E\int_0^T \norm[\big]{G_i(s)}^2_{\cL^2(H,L^2)}\,ds < \infty,
  \qquad i=1,2.
  \]
  If $u_i \in \mathscr{K}$, $i=1,2$, are solutions to
  \[
  du_i - \div \gamma(\nabla u_i)\,dt = G_i\,dW,
  \qquad u_i(0)=u_{0i},
  \]
  we are going to show that
  \begin{equation}
  \label{eq:lippo}
  \E\sup_{t\leq T} \norm[\big]{u_1(t) - u_2(t)}^2 \lesssim
  \E\norm[\big]{u_{01}-u_{02}}^2
  + \E\int_0^T \norm[\big]{G_1(s)-G_2(s)}^2_{\cL^2(H,L^2)}\,ds,
  \end{equation}
  from which uniqueness and Lipschitz-continuous dependence on the
  initial datum follow immediately. We shall actually obtain this
  estimate as a special case of a more general one that will be
  useful in the next section: setting
  \[
  y(t) := u_1(t) - u_2(t), \qquad y_0 := u_{01} - u_{02}, \qquad
  F(t) := G_1(t) - G_2(t),
  \]
  one has
  \[
  y(t) - \int_0^t \div\zeta(s)\,ds = y_0 + \int_0^t F(s)\,dW(s),
  \]
  where $\zeta = \gamma(\nabla u_1) - \gamma(\nabla u_2)$. Setting,
  for any $\alpha \geq 0$,
  \[
  y^\alpha(t) := e^{-\alpha t} y(t), \qquad
  \zeta(t) := e^{-\alpha t} \zeta(t), \qquad
  F^\alpha(t) := e^{-\alpha t} F(t),
  \]
  the integration by parts formula yields
  \[
  y^{\alpha}(t) 
  + \int_0^t \bigl( \alpha y^{\alpha}(s) - \div\zeta^{\alpha}(s) \bigr)\,ds 
  = y_0 + \int_0^t F^{\alpha}(s)\,dW(s),
  \]
  from which, by Proposition~\ref{prop:ito}, we deduce
  \begin{align*}
  &\norm[\big]{y^{\alpha}(t)}^2
  + 2\alpha \int_0^t \norm[\big]{y^{\alpha}(s)}^2\,ds
  + 2 \int_0^t\!\!\int_D \zeta^{\alpha}(s)\cdot\nabla y^{\alpha}(s)\,ds\\
  &\hspace{3em} \leq \norm[\big]{y_0}^2
  +2 \int_0^t y^{\alpha}(s) F^{\alpha}(s)\,dW(s)
  + \int_0^t \norm[\big]{F^{\alpha}(s)}_{\cL^2(H,L^2)}^2\,ds,
  \end{align*}
  where, by monotonicity of $\gamma$, $\zeta^\alpha \cdot \nabla
  y^\alpha = e^{-2\alpha \cdot} \bigl(\gamma(\nabla u_1) - \gamma(\nabla
  u_2)\bigr) \cdot (\nabla u_1-\nabla u_2) \geq 0$. 
  Therefore, taking the supremum in $t$ and expectation on both
  sides, one has
  \begin{align}
  &\E\sup_{t\leq T} \norm[\big]{y^\alpha(t)}^2 +
  \alpha \E\int_0^T \norm[\big]{y^\alpha(s)}^2\,ds\nonumber \\
  &\hspace{3em} \lesssim
  \E\norm[\big]{y_0}^2 + \E\sup_{t\leq T}
  \abs[\bigg]{\int_0^t y^\alpha(s) F^\alpha(s)\,dW(s)}
  + \E\int_0^T \norm[\big]{F^\alpha(s)}^2_{\cL^2(H,L^2)}\,ds\nonumber\\
  &\hspace{3em} \lesssim
  \E\norm[\big]{y_0}^2
  + \E\int_0^T \norm[\big]{F^\alpha(s)}^2_{\cL^2(H,L^2)}\,ds,
  \label{eq:clp}
  \end{align}
  where the second inequality follows by an application of
  Lemma~\ref{lm:DY}. Estimate \eqref{eq:lippo} is just the special
  case $\alpha=0$.
\end{proof}


\ifbozza\newpage\else\fi
\section{Proof of the main result}
\label{sec:pf}
Thanks to the results established thus far, we are now in the position
to prove Theorem~\ref{thm:main}.
Let $v:\Omega \times [0,T] \to L^2$ be a measurable adapted process
such that
\[
\E\int_0^T \norm[\big]{v(s)}^2\,ds < \infty,
\]
and consider the equation
\[
du(t) -\div\gamma(\nabla u(t))\,dt = B(t,v(t))\,dW(t),
\qquad u(0)=u_0,
\]
where $u_0$ is an $\mathscr{F}_0$-measurable $L^2$-valued random
variable with finite second moment.
The assumptions on $B$ imply that $B(\cdot,v)$ is measurable, 
adapted, and such that
\[
\E\int_0^T \norm[\big]{B(s,v(s))}^2_{\cL^2(H,L^2)}\,ds < \infty,
\]
hence the above equation is well-posed in $\mathscr{K}$ by Proposition
\ref{prop:add}, which allows one to define a map
$\Gamma:(u_0,v) \mapsto u$. Let $u_i=\Gamma(u_{0i},v_i)$, $i=1,2$,
where $u_{0i}$ and $v_i$ satisfy the same measurability and
integrability assumptions on $u_0$ and $v$, respectively. For any
$\alpha \geq 0$, \eqref{eq:clp} and the Lipschitz continuity of $B$ yield
\begin{align*}
  &\E\sup_{t\leq T} \bigl( e^{-2\alpha t} \norm[\big]{u_1(t) - u_2(t)}^2 \bigr)
   + \E\int_0^T e^{-2\alpha s} \norm[\big]{u_1(s)-u_2(s)}^2\,ds\\
  &\hspace{3em} \lesssim
  \frac{1}{\alpha} \E\norm[\big]{u_{01}-u_{02}}^2
  + \frac{1}{\alpha} \E\int_0^T e^{-2\alpha s}%
   \norm[\big]{B(s,v_1(s))-B(s,v_2(s))}^2_{\cL^2(H,L^2)}\,ds\\
  &\hspace{3em} \lesssim
  \frac{1}{\alpha} \E\norm[\big]{u_{01}-u_{02}}^2
  + \frac{1}{\alpha} \E\int_0^T e^{-2\alpha s}%
   \norm[\big]{v_1(s) - v_2(s)}^2\,ds.
\end{align*}
Choosing $\alpha$ large enough, it follows that, for any $u_0$ as
above, the map $v \mapsto \Gamma(u_0,v)$ is strictly contractive in
the Banach space $E_\alpha$ of measurable adapted processes $v$ such
that
\[
  \norm{v}_{E_\alpha} := \biggl( \E\int_0^T e^{-2\alpha s}
  \norm{v(s)}^2\,ds \biggr)^{1/2}.
\]
By the Banach fixed point theorem, the map $v \mapsto \Gamma(u_0,v)$
admits a unique fixed point $u$ in $E_\alpha$. Since all
$E_\alpha$-norms are equivalent for different values of $\alpha$, $u$
belongs to $E_0$ and, by definition of $\Gamma$, $u$ also belongs to
$\mathscr{K}$ and solves \eqref{eq:0}. Taking into account that
any solution to \eqref{eq:0} is necessarily a fixed point of $v
\mapsto \Gamma(u_0,v)$, it immediately follows that $u$ is the unique
solution to \eqref{eq:0} in $\mathscr{K}$. Lipschitz continuity of the
solution map follows from the above estimate, which manifestly implies
\[
\E\int_0^T \norm[\big]{u_1(s) - u_2(s)}^2\,ds \eqsim
\E\int_0^T e^{-2\alpha s} \norm[\big]{u_1(s) - u_2(s)}^2\,ds
\lesssim \E\norm[\big]{u_{01}-u_{02}}^2.
\]
and
\begin{align*}
\E\sup_{t\leq T} \norm[\big]{u_1(t) - u_2(t)}^2 &\eqsim
\E\sup_{t\leq T} \bigl( e^{-2\alpha t} \norm[\big]{u_1(t) - u_2(t)}^2 \bigr)\\
&\lesssim \E\int_0^T e^{-2\alpha s} \norm[\big]{u_1(s) - u_2(s)}^2\,ds,
\end{align*}
thus completing the proof.


\ifbozza\newpage\else\fi
\appendix
\section{A remark on uniform integrability}
The classical characterization of uniform integrability by de la
Vall\'ee Poussin states that, in the setting of a measure space
$(X,\mathcal{A})$ endowed with a finite measure $\mu$, a bounded
subset $\mathscr{G}$ of $L^1(X,\mu;\erre^n)$ is uniformly integrable
if and only if there exists a continuous increasing convex function
$\varphi:\erre_+ \to \erre_+$, with $\varphi(0)=0$ and
$\lim_{x\to\infty}\varphi(x)/x=\infty$, such that
\[
\int_A \varphi(\abs{g})\,d\mu < 1 \qquad \forall g \in \mathscr{G}
\]
(see, e.g., \cite[p.~12]{AFP-BV}).

\medskip

The following criterion for uniform integrability can be proved in the
same way (the proof is included for completeness).
\begin{lemma}
  Let $F:\erre^n \to \erre_+$ be a continuous convex function such
  that $F(0)=0$ and
  \[
  \lim_{\abs{x} \to \infty} \frac{F(x)}{\abs{x}} = \infty.
  \]
  Let $\mathscr{G}$ be a subset of $L^0(X,\mu;\erre^n)$ such that
  $F(\mathscr{G})$ is bounded in $L^1(X,\mu)$. Then $\mathscr{G}$
  is uniformly integrable.
\end{lemma}
\begin{proof}
  We have to prove that $\mathscr{G}$ is bounded in $L^1(X,\mu)$ and
  that for any $\varepsilon>0$ there exists $\delta>0$ such that, for
  any $A \in \mathcal{A}$ with $\mu(A) < \delta$,
  \[
  \int_A \abs{g}\,d\mu < \varepsilon \qquad \forall g \in \mathscr{G}.
  \]
  By definition of limit, for any $M>0$ there exists $R$ (depending on
  $M$) such that $\abs{x}<F(x)/M$ for all $x\in\erre^n$ such that
  $\abs{x}>R$. Then
  \begin{align*}
  \int_A \abs{g}\,d\mu &= \int_{A \cap \{\abs{g} \leq R\}} \abs{g}\,d\mu 
  + \int_{A \cap \{\abs{g}>R\}} \abs{g}\,d\mu\\
  &\leq R\mu(A) + \frac{1}{M} \int_X F(g)\,d\mu
  \end{align*}
  for all $g \in \mathscr{G}$. Choosing $A=X$, this proves that
  $\mathscr{G}$ is bounded in $L^1(X,\mu)$. Let $\varepsilon>0$ be
  arbitrary, and choose $M$ such that the second-term on the
  right-hand side of the last inequality is smaller than
  $\varepsilon/2$. Then $\delta:=\varepsilon/(2R)$ satisfies the
  required condition.
\end{proof}

\ifbozza\newpage\else\fi

\def\polhk#1{\setbox0=\hbox{#1}{\ooalign{\hidewidth
  \lower1.5ex\hbox{`}\hidewidth\crcr\unhbox0}}}
\providecommand{\bysame}{\leavevmode\hbox to3em{\hrulefill}\thinspace}
\providecommand{\MR}{\relax\ifhmode\unskip\space\fi MR }
\providecommand{\MRhref}[2]{%
  \href{http://www.ams.org/mathscinet-getitem?mr=#1}{#2}
}
\providecommand{\href}[2]{#2}

\end{document}